\documentclass[reqno]{amsart}
\usepackage{amssymb,amsmath,hyperref}
\usepackage{amsrefs, float}
\usepackage[foot]{amsaddr}
\usepackage{bbold,stackrel, multicol}
\usepackage{mathrsfs}  
 
\newcommand{\Z}{{\textsf{\textup{Z}}}}

\newtheorem{thm}{Theorem}

\newtheorem{cor}[thm]{Corollary}
\newtheorem{defi}[thm]{Definition}
\newtheorem{rem}[thm]{Remark}
\newtheorem{nota}[thm]{Notation}

\newtheorem{princ}[thm]{Principle}

\newtheorem{ack}[thm]{Acknowledgement}

\newtheorem*{tempo*}{Template}

\newcommand\be{\begin{equation}}
\newcommand\ee{\end{equation}} 

\usepackage{amsmath,amsfonts} 

\usepackage[applemac]{inputenc}

\def\bdefi{\begin{defi}}
\def\edefi{\end{defi}}
\def\bnota{\begin{nota}\rm}
\def\enota{\end{nota}}
\def\FIVE{\Pi_{1}^{1}\text{-\textup{\textsf{CA}}}_{0}}
\def\SIX{\Pi_{2}^{1}\text{-\textsf{\textup{CA}}}_{0}}

\def\SIXK{\Pi_{k}^{1}\text{-\textsf{\textup{CA}}}_{0}^{\omega}}

\def\ZFC{\textup{\textsf{ZFC}}}
\def\ZF{\textup{\textsf{ZF}}}

\def\osc{\textup{\textsf{osc}}}

\def\L{\textsf{\textup{L}}}

 \def\r{\mathbb{r}}

\def\RCA{\textup{\textsf{RCA}}}
\def\({\textup{(}}
\def\){\textup{)}}

\def\RCAo{\textup{\textsf{RCA}}_{0}^{\omega}}
\def\ACAo{\textup{\textsf{ACA}}_{0}^{\omega}}

\def\WKL{\textup{\textsf{WKL}}}

\def\bye{\end{document}}
\def\N{{\mathbb  N}}
\def\Q{{\mathbb  Q}}
\def\R{{\mathbb  R}}

\def\SS{\textup{\textsf{S}}}

\def\di{\rightarrow}

\def\asa{\leftrightarrow}

\def\ACA{\textup{\textsf{ACA}}}

\def\QFAC{\textup{\textsf{QF-AC}}}
\def\AC{\textup{\textsf{AC}}}

\def\ACG{\textup{\textsf{ACG}}}
\def\PHP{\textup{\textsf{PHP}}}

\def\osc{\textup{\textsf{osc}}}
\def\FIN{\textup{\textsf{FIN}}}
\def\alt{\textup{\textsf{alt}}}

\def\NCC{\textup{\textsf{NCC}}}
\def\MCC{\textup{\textsf{MCC}}}

\def\INDY{\textup{\textsf{IND}}_{0}}

\def\NIN{\textup{\textsf{NIN}}}
\def\NCC{\textup{\textsf{NCC}}}

\def\BCT{\textup{\textsf{BCT}}}
\def\WBCT{\textup{\textsf{WBCT}}}

\def\open{\textup{\textsf{open}}}

\def\IND{\textup{\textsf{IND}}}

\def\eps{\varepsilon}

\def\ECF{\textup{\textsf{ECF}}}

\newcommand{\F}{{\bf F}}

\usepackage{graphicx}
\usepackage{tikz}
\usetikzlibrary{matrix}
\usepackage{comment,tikz-cd}

\numberwithin{equation}{section}
\numberwithin{thm}{section}

\begin{document}
\title{Big in Reverse Mathematics: Measure and Category}
\author{Sam Sanders}
\address{Department of Philosophy II, RUB Bochum, Germany}
\email{sasander@me.com}
\begin{abstract}
The smooth development of large parts of mathematics hinges on the idea that some sets are `small' or `negligible' and can therefore
be ignored for a given purpose.  The perhaps most famous smallness notion, namely `measure zero', originated with Lebesgue, while a second smallness notion, namely `meagre' or `first category', originated 
with Baire around the same time.  The associated \emph{Baire category theorem} is a central result governing the properties of meagre (and related) sets, while the same holds for Tao's \emph{pigeonhole principle for measure spaces} and measure zero sets.  
In this paper, we study these theorems in Kohlenbach's \emph{higher-order Reverse Mathematics}, identifying a considerable number of equivalent and robust theorems.
The latter involve most basic properties of semi-continuous and pointwise discontinuous functions, Blumberg's theorem, Riemann integration, and Volterra's early work circa 1881. 
All the aforementioned theorems fall (far) outside of the \emph{Big Five} of Reverse Mathematics, and we investigate natural restrictions like Baire 1 and quasi-continuity that make these theorems provable again in the Big Five (or similar).  Finally, despite the fundamental differences between measure and category, the proofs of our equivalences turn out to be similar.
\end{abstract}


\maketitle
\thispagestyle{empty}

\section{Introduction}\label{intro}
In nutshell, we study basic theorems concerning measure and category in Kohlenbach's \emph{higher-order} Reverse Mathematics (RM for short; see Section \ref{prelim}).  
We obtain numerous equivalences for the \emph{Baire category theorem} and Tao's \emph{pigeonhole principle for measure spaces}, involving basic properties of \emph{semi-continuous}, \emph{Riemann integrable}, and \emph{pointwise discontinuous} functions, where the first notion is fairly tame and the latter is wild, as discussed in Remark \ref{donola}.  
An overview of our results is in Section~\ref{vower} while preliminaries may be found in Section~\ref{helim}.
We discuss Volterra's early work from 1881 in Section~\ref{vintro}, as it pertains to this paper.   
\subsection{Summary and background}\label{vower}
We provide some background for and an overview of the results to be proved in Sections \ref{mainn} and \ref{related}. 

\smallskip

First of all, it is a commonplace that the smooth development of large parts of mathematics hinges on the idea that some sets are `small' or `negligible' and can therefore
be ignored for a given purpose.  While some calculus students may not (ever) realise it, 
even introductory calculus depends (perhaps indirectly) on a specific smallness notion, in light of the \emph{Vitali-Lebesgue theorem} as follows.
\begin{thm}[Vitali-Lebesgue]
{A function on the unit interval is Riemann integrable if and only if it is continuous {almost everywhere} and bounded. }
\end{thm}
The associated smallness notion `measure zero' originated with Lebesgue (\cite{lebesborn}) and is perhaps the most famous\footnote{A less famous smallness notion is \emph{porosity} (see \cite{poro}), which we may study in a future paper.} such notion.  
A second smallness notion, namely `meagre' or `first category', originated with Baire around the same time (\cite{beren2}).  
Tao discusses the connection between these two notions in \cite{taoeps}*{\S1.7}.  As pointed out by Smith in \cite{snutg}, the conflation of measure and category 
led to a number of incorrect results by Hankel (\cite{hankelwoot}) and others regarding the Riemann integral.   
In this context, Smith introduced the famous \emph{Cantor set} some years before Cantor did (\cite{snutg}).  

\smallskip

Secondly, the \emph{Baire category theorem} is a central result governing the properties of meagre sets (see  \cite{taokejes}*{\S3.2}), while Tao's \emph{pigeonhole principle for measure spaces} plays the same role for measure zero sets (\cite{taoeps}*{\S1.7}).  We shall use the obvious abbreviations $\BCT_{[0,1]}$ and $\PHP_{[0,1]}$ in the below.  
By Theorems \ref{tonko} and~\ref{pinko},  the third-order principles $\BCT_{[0,1]}$ and $\PHP_{[0,1]}$ are \emph{hard to prove} in terms of
conventional\footnote{The system $\Z_{2}^{\omega}$ from Section \ref{lll} is a conservative extension of $\Z_{2}$ via third-order comprehension functionals $\SS_{k}^{2}$ that decide $\Pi_{k}^{1}$-formulas.  By \cite{dagsamVII}*{Theorem 6.16} or \cite{dagsamX}*{Theorem 3.4}, $\Z_{2}^{\omega}+\QFAC^{0,1}$ cannot prove $\BCT_{[0,1]}$; a similar result is proved in Theorem \ref{tonko}.\label{kabal}} comprehension.  
Similar hardness results (see \cite{dagsamVII}) exist for Kleene's computability theory based on S1-S9 (\cite{longmann}); moreover, we have established a large number of (S1-S9) \emph{computationally} equivalent principles for $\BCT_{[0,1]}$ in \cite{samcsl23}.  By contrast, $\BCT_{[0,1]}$ and $\PHP_{[0,1]}$ \emph{restricted to second-order codes for closed sets} (see \cite{simpson2}*{II.5.6}) are provable in relatively weak systems by \cite{simpson2}*{II.5.8} and Theorem \ref{pinko}.

\smallskip

Thirdly, in light of the above, the RM-study of $\BCT_{[0,1]}$ and $\PHP_{[0,1]}$ is an important and natural enterprise, and constitutes the topic of this paper.
In particular, we show in Section \ref{mainn} that $\BCT_{[0,1]}$ is equivalent to the well-known principles listed below, working in {higher-order} RM (see Section~\ref{prelim}). 
The notion of \emph{semi-continuity} goes back to Baire (\cite{beren}) while the notion of \emph{pointwise discontinuity} goes back to Hankel (\cite{hankelwoot}) and is equivalent to \emph{cliquishness} (see Section \ref{cdef} for definitions).
\begin{enumerate}
\renewcommand{\theenumi}{\roman{enumi}}
\item For upper semi-continuous (usco for short) $f:[0,1]\di \R$, there is a point $x\in [0,1]$ where $f$ is continuous (or: quasi-continuous, or: Darboux).\label{fil1}
\item For Baire $1^{*}$ $f:[0,1]\di \R$, there is a point $x\in [0,1]$ where $f$ is continuous.\label{fil2}
\item For fragmented $f:[0,1]\di \R$, there is $x\in [0,1]$ where $f$ is continuous.\label{filx}
\item For cliquish $f:[0,1]\di \R$, there is a point $x\in [0,1]$ where $f$ is continuous.\label{fil3}
\item (Volterra \cite{volaarde2}) For cliquish $f,g:[0,1]\di \R$, there is $x\in [0,1]$ such that $f$ and $g$ are both continuous or both discontinuous at $x$. \label{volkert1}
\item (Volterra \cite{volaarde2}) For cliquish $f:[0,1]\di \R$, there is either $q\in \Q\cap [0,1]$ where $f$ is discontinuous, or $x\in [0,1]\setminus \Q$ where $f$ is continuous.\label{volkert2}
\item The previous two items restricted to usco functions. 
\item Blumberg's theorem (\cite{bloemeken}) restricted to cliquish (or: usco) functions. \label{fil8}
\end{enumerate}
Some of the above theorems stem from Volterra's early work (1881) in the spirit of -but predating- the Baire category theorem, as discussed in Section~\ref{vintro}.  
The (necessity for the) absence of `Baire 1' from the above items is explained at the end of this section.  
In Section \ref{related}, we obtain equivalences for $\PHP_{[0,1]}$, including the following well-known principles, like the aforementioned Vitali-Lebesgue theorem.
\begin{enumerate}
\renewcommand{\theenumi}{\roman{enumi}}
\setcounter{enumi}{8}
\item \(Vitali-Lebesgue\) For Riemann integrable $f:[0,1]\di \R$, the set of continuity points $C_{f}$ has measure one.\label{tao1X}
\item For Riemann integrable $f:[0,1]\di [0,1]$ such that $\int_{0}^{1}f(x)dx=0$, the set $\{x\in [0,1]:f(x)=0\}$ has measure one.\label{taoX3}

\item \textsf{\textup{(FTC)}} For Riemann integrable $f:[0,1]\di \R$ with $F(x):=\lambda x.\int_{0}^{x}f(t)dt$, the following set exists and has measure one: 
\[
\{x\in [0,1]:F \textup{ is differentiable at $x$ with derivative } f(x)\}
\]
\item The previous item restricted to usco or cliquish functions.\label{taoX7}
\end{enumerate}
We obtain similar equivalences for $\WBCT_{[0,1]}$, a `hybrid' version of $\BCT_{[0,1]}$ and $\PHP_{[0,1]}$, with equivalences that are interesting in their own right.  
The relation between these and related principles may be found in Figure \ref{xxx} in Section \ref{reflm}. 

\smallskip
\noindent
Fourth, the above items suggests that our results are \emph{robust} as follows:
\begin{quote}
A system is \emph{robust} if it is equivalent to small perturbations of itself. (\cite{montahue}*{p.\ 432}; emphasis in original)
\end{quote}
Most of our results shall be seen to exhibit a similar (or stronger) level of robustness.  For instance, item \eqref{fil3} above can be restricted to any Baire class $\alpha\geq 2$ and the equivalence still goes through.  
In this light, we feel that $\BCT_{[0,1]}$ and $\PHP_{[0,1]}$ deserve the moniker `big' system in the way this notion is used in second-order RM, namely as boasting many equivalences from different fields of mathematics. 

\smallskip

Fifth, items \eqref{fil1}-\eqref{taoX7} are \emph{hard to prove} in terms of conventional comprehension, in the sense of Footnote \ref{kabal}.  
We shall identify natural extra conditions that render the above items
provable in terms of arithmetical comprehension (or at least the Big Five), like quasi-continuity (resp.\ Baire 1) for item \eqref{fil3} (resp.\ item \eqref{fil1}).
As discussed in Remark \ref{donola}, quasi-continuity and cliquishness are however closely related. 
Moreover, usco (and fragmented) functions are of course Baire 1, say over $\ZF$; 
there is no contradiction here as e.g.\ the statement 
\begin{center}
\emph{a bounded usco function on the unit interval is Baire 1} 
\end{center}
already implies (stronger principles than) $\BCT_{[0,1]}$ by Corollary~\ref{lopi}.  This explains why the above items~\eqref{fil1}-\eqref{taoX7} do not (and cannot) deal with Baire 1 or quasi-continuous functions: the latter conditions make the associated theorems provable in relatively weak systems, like arithmetical comprehension.  

\smallskip

Finally, we discuss the bigger picture relating to $\PHP_{[0,1]}$ and $\BCT_{[0,1]}$ in Section~\ref{bigger}. 
We argue that the RM of the uncountability of $\R$ is the product of taking the RM of $\PHP_{[0,1]}$ and the RM of $\BCT_{[0,1]}$ and pushing everything down from usco/cliquish functions to the level of regulated/bounded variation functions.  This explains why the RM of $\PHP_{[0,1]}$ and the RM of $\BCT_{[0,1]}$ are similar, despite the fundamental differences between measure and category. 
We also explore the connections between our results and set theory (Section \ref{biggerer}), and show (Section \ref{nex}) that \emph{slight} variations of Banach's theorem, i.e.\ that the continuous nowhere differentiable are dense in $C([0,1])$, are hard to prove in the sense of Footnote \ref{kabal}. 

\subsection{Volterra's early work and related results}\label{vintro}
We introduce Volterra's early work from \cite{volaarde2}, called a \emph{historical gem} in \cite{dendunne}, and related results.

\smallskip

First of all, the Riemann integral was groundbreaking for a number of reasons, including its ability to integrate functions with infinitely many points of discontinuity, as shown by Riemann himself (\cite{riehabi}). 
A natural question is then `how discontinuous' a Riemann integrable function can be.  In this context, Thomae introduced the function $T:\R\di\R$ around 1875 in \cite{thomeke}*{p.\ 14, \S20} as follows:
\be\label{thomae}\tag{\textup{\textsf{T}}}
T(x):=
\begin{cases} 
0 & \textup{if } x\in \R\setminus\Q\\
\frac{1}{q} & \textup{if $x=\frac{p}{q}$ and $p, q$ are co-prime} 
\end{cases}.
\ee
Thomae's function $T$ is integrable on any interval, but has a dense set of points of discontinuity, namely $\Q$, and a dense set of points of continuity, namely $\R\setminus \Q$. 

\smallskip

The perceptive student, upon seeing Thomae's function as in \eqref{thomae}, will ask for a function continuous at each rational point and discontinuous at each irrational one.
Such a function cannot exist, as is generally proved using the Baire category theorem.  
However, Volterra in \cite{volaarde2} already established this negative result about twenty years before the publication of the Baire category theorem.

\smallskip

Secondly, as to the content of Volterra's paper \cite{volaarde2}, we find the following theorem, where a function is \emph{pointwise discontinuous} if it has a dense set of continuity points.
On the real line, this is equivalent to being \emph{cliquish} (see Section \ref{cdef}).
\begin{thm}[Volterra, 1881]\label{VOL}
There do not exist pointwise discontinuous functions defined on an interval for which the continuity points of one are the discontinuity points of the other, and vice versa.
\end{thm}
Volterra then states two corollaries, of which the following is perhaps well-known in `popular mathematics' and constitutes the aforementioned negative result. 
\begin{cor}[Volterra, 1881]\label{VOLcor}
There is no $\R\di\R$ function that is continuous on $\Q$ and discontinuous on $\R\setminus\Q$. 
\end{cor}
We note that pointwise discontinuous functions were already studied by Dini before 1878, including an equivalent definition (see \cite{dinipi}*{\S63}) that amounts to \emph{cliquishness} (see Section \ref{cdef}) and the observation that Riemann integrable functions are pointwise discontinuous, following Hankel (see \cite{dinipi}*{\S188} and \cite{hankelwoot}).

\smallskip

Thirdly,  Volterra's results from \cite{volaarde2} are generalised in \cite{volterraplus,gaud}.  
The following theorem is immediate from these generalisations. 
\begin{thm}\label{dorki}
For any countable dense set $D\subset [0,1]$ and $f:[0,1]\di \R$, either $f$ is discontinuous at some point in $D$ or continuous at some point in $[0,1]\setminus D$. 
\end{thm}
A related result is as follows, which we will study for e.g.\ cliquish functions. 
\begin{thm}[Blumberg's theorem, \cite{bloemeken}]
For any $f:\R\di \R$, there is a dense subset $D\subset \R$ such that the restriction of $f$ to $D$, usually denoted $f_{\upharpoonright D}$, is continuous.  
\end{thm}
\noindent
To be absolutely clear, the conclusion of Blumberg's theorem means that 
\[
(\forall x\in D, \eps>0)(\exists \delta>0)\underline{(\forall y\in D)}(|x-y|<\delta\di |f(x)-f(y)|<\eps)), 
\]
where the underlined quantifier marks the difference with `usual' continuity.

\smallskip

Fourth, the Baire category theorem for the real line was first proved by Osgood (\cite{fosgood}) and later by Baire (\cite{beren2}) in a more general setting.
\begin{thm}[Baire category theorem]\label{konkli}
If $ (O_n)_{n \in \N}$ is a sequence of dense open sets of reals, then 
$ \bigcap_{n \in\N } O_n$ is non-empty.
\end{thm}
We have previously studied the above theorem restricted to the unit interval in \cite{dagsamVII}, but failed to obtain any equivalences.  
In Section \ref{mainn}, we obtain a number of equivalences for the Baire category theorem.  
To this end, we will need some preliminaries and definitions, as in Section~\ref{helim}.

\smallskip

Fifth, Baire has shown that separately continuous $\R^{2}\di \R$ are \emph{quasi-continuous} in one variable (see Section \ref{cdef} for the latter).  
Baire mentions in \cite{beren2}*{p.\ 95} that the latter notion (without naming it) was suggested to him by Volterra.  
The naming and independent study was done later by Kempisty in \cite{kemphaan}.

\smallskip

Finally, in light of our definition of open set in Definition \ref{char}, our study of the Baire category theorem (and the same for \cite{dagsamVII})
is based on a most general definition of open set, i.e.\ no additional (computational) information is given.  Besides the intrinsic interest of such an investigation,
there is a deeper reason, as discussed in Remark \ref{dichtbij}, namely that open sets without any representation are readily encountered `in the wild'.

\subsection{Preliminaries and definitions}\label{helim}
We briefly introduce the program \emph{Reverse Mathematics} in Section \ref{prelim}.
We introduce some essential axioms (Section \ref{lll}) and definitions (Section~\ref{cdef}).  
Our RM-study shall make (sometimes essential) use of \emph{oscillation functions}, a construct which goes back to Riemann and Hankel, as discussed in Section \ref{prebairen}.

\subsubsection{Reverse Mathematics}\label{prelim}
Reverse Mathematics (RM hereafter) is a program in the foundations of mathematics initiated around 1975 by Friedman (\cites{fried,fried2}) and developed extensively by Simpson (\cite{simpson2}).  
The aim of RM is to identify the minimal axioms needed to prove theorems of ordinary, i.e.\ non-set theoretical, mathematics. 

\smallskip

We refer to \cite{stillebron} for a basic introduction to RM and to \cite{simpson2, simpson1,damurm} for an overview of RM.  We expect basic familiarity with RM, in particular Kohlenbach's \emph{higher-order} RM (\cite{kohlenbach2}) essential to this paper, including the base theory $\RCAo$.   An extensive introduction can be found in e.g.\ \cites{dagsamIII, dagsamV, dagsamX} and elsewhere.  

\smallskip

We have chosen to include a brief introduction as a technical appendix, namely Section \ref{RMA}.  
All undefined notions may be found in the latter, while we do point out here that we shall sometimes use common notations from type theory.  For instance, the natural numbers are type $0$ objects, denoted $n^{0}$ or $n\in \N$.  
Similarly, elements of Baire space are type $1$ objects, denoted $f\in \N^{\N}$ or $f^{1}$.  Mappings from Baire space $\N^{\N}$ to $\N$ are denoted $Y:\N^{\N}\di \N$ or $Y^{2}$.

\subsubsection{Some comprehension functionals}\label{lll}
In second-order RM, the logical hardness of a theorem is measured via what fragment of the comprehension axiom is needed for a proof.  
For this reason, we introduce some axioms and functionals related to \emph{higher-order comprehension} in this section.
We are mostly dealing with \emph{conventional} comprehension here, i.e.\ only parameters over $\N$ and $\N^{\N}$ are allowed in formula classes like $\Pi_{k}^{1}$ and $\Sigma_{k}^{1}$.  

\smallskip
\noindent
First of all, the functional $\varphi$ in $(\exists^{2})$ is also \emph{Kleene's quantifier $\exists^{2}$} and is clearly discontinuous at $f=11\dots$ in Cantor space:
\be\label{muk}\tag{$\exists^{2}$}
(\exists \varphi^{2}\leq_{2}1)(\forall f^{1})\big[(\exists n^{0})(f(n)=0) \asa \varphi(f)=0    \big]. 
\ee
In fact, $(\exists^{2})$ is equivalent to the existence of $F:\R\di\R$ such that $F(x)=1$ if $x>_{\R}0$, and $0$ otherwise (see \cite{kohlenbach2}*{Prop.\ 3.12}).  
Related to $(\exists^{2})$, the functional $\mu^{2}$ in $(\mu^{2})$ is called \emph{Feferman's $\mu$} (see \cite{avi2}) and may be found -with the same symbol- in Hilbert-Bernays' Grundlagen (\cite{hillebilly2}*{Supplement IV}):
\begin{align}\label{mu}\tag{$\mu^{2}$}
(\exists \mu^{2})(\forall f^{1})\big[ (\exists n)(f(n)=0) \di [f(\mu(f))=0&\wedge (\forall i<\mu(f))(f(i)\ne 0) ]\\
& \wedge [ (\forall n)(f(n)\ne0)\di   \mu(f)=0]    \big]. \notag
\end{align}
We have $(\exists^{2})\asa (\mu^{2})$ over $\RCAo$ (see \cite{kohlenbach2}*{\S3}) and $\ACAo\equiv\RCAo+(\exists^{2})$ proves the same sentences as $\ACA_{0}$ by \cite{hunterphd}*{Theorem~2.5}. 

\smallskip

Secondly, the functional $\SS^{2}$ in $(\SS^{2})$ is called \emph{the Suslin functional} (\cite{kohlenbach2}):
\be\tag{$\SS^{2}$}
(\exists\SS^{2}\leq_{2}1)(\forall f^{1})\big[  (\exists g^{1})(\forall n^{0})(f(\overline{g}n)=0)\asa \SS(f)=0  \big].
\ee
The system $\FIVE^{\omega}\equiv \RCAo+(\SS^{2})$ proves the same $\Pi_{3}^{1}$-sentences as $\FIVE$ by \cite{yamayamaharehare}*{Theorem 2.2}.   
By definition, the Suslin functional $\SS^{2}$ can decide whether a $\Sigma_{1}^{1}$-formula as in the left-hand side of $(\SS^{2})$ is true or false.   We similarly define the functional $\SS_{k}^{2}$ which decides the truth or falsity of $\Sigma_{k}^{1}$-formulas from $\L_{2}$; we also define 
the system $\SIXK$ as $\RCAo+(\SS_{k}^{2})$, where  $(\SS_{k}^{2})$ expresses that $\SS_{k}^{2}$ exists.  
We note that the operators $\nu_{n}$ from \cite{boekskeopendoen}*{p.\ 129} are essentially $\SS_{n}^{2}$ strengthened to return a witness (if existant) to the $\Sigma_{n}^{1}$-formula at hand.  

\smallskip

\noindent
Thirdly, full second-order arithmetic $\Z_{2}$ is readily derived from $\cup_{k}\SIXK$, or from:
\be\tag{$\exists^{3}$}
(\exists E^{3}\leq_{3}1)(\forall Y^{2})\big[  (\exists f^{1})(Y(f)=0)\asa E(Y)=0  \big], 
\ee
and we therefore define $\Z_{2}^{\Omega}\equiv \RCAo+(\exists^{3})$ and $\Z_{2}^\omega\equiv \cup_{k}\SIXK$, which are conservative over $\Z_{2}$ by \cite{hunterphd}*{Cor.\ 2.6}. 
Despite this close connection, $\Z_{2}^{\omega}$ and $\Z_{2}^{\Omega}$ can behave quite differently, as discussed in e.g.\ \cite{dagsamIII}*{\S2.2}.   
The functional from $(\exists^{3})$ is also called `$\exists^{3}$', and we use the same convention for other functionals.  

\smallskip

Finally, the following negative results were established in \cite{dagsamX, dagsamXI} using the technique \emph{Gandy selection} from Kleene computability theory.
\begin{center}
\emph{Neither $\Z_{2}^{\omega}+\QFAC^{0,1}$ nor $\Z_{2}^{\omega}+\IND_{0}$ can prove $\NIN_{[01]}$, while $\Z_{2}^{\Omega}$ can.}
\end{center}
Here, $\NIN_{[0,1]}$ states that there is no injection from $[0,1]$ to $\N$ as follows:
\[
(\forall Y:[0,1]\di \N)(\exists x, y\in [0,1])(x\ne_{\R} y\wedge Y(x)=_{\N}Y(y)),
\]
and $\IND_{0}$ is the following fragment of the induction axiom. 
\bdefi[$\INDY$]
Let $Y^{2}$ satisfy $(\forall n\in \N)(\exists \textup{ at most one } f\in 2^{\N})(Y(f, n)=0)$.  
For $k\in \N$, there is $w^{1^{*}}$ with $|w|=k$ such that for $m\leq k$, we have:
\[
(w(m)\in 2^{\N}\wedge Y(w(m), m)=0) \asa (\exists f\in 2^{\N})(Y(f, m)=0).
\]
\edefi
A limited number of equivalences for $\IND_{0}$ (resp.\ $\NIN_{[0,1]}$) may be found in \cite{dagsamX}*{\S3} (resp.\ \cite{samcie22}).
A large number of equivalences for a slight \emph{variation} of $\NIN_{[0,1]}$ may be found in \cite{samBIG}, as also discussed in Section \ref{bigger}.

\subsubsection{Some definitions}\label{cdef}
We introduce some definitions needed in the below, mostly stemming from mainstream mathematics.
We assume the standard `epsilon-delta' definitions of continuity and Riemann integrability to be known.
We note that subsets of $\R$ are given by their characteristic functions as in Definition \ref{char}, well-known from measure and probability theory.
We shall generally work over $\ACAo$ as some definitions make little sense over $\RCAo$.

\smallskip

First of all, we make use the usual definition of (open) set, where $B(x, r)$ is the open ball with radius $r>0$ centred at $x\in \R$.
 We note that our notion of `measure zero' does not depend on (the existence of) the Lebesgue measure.
\bdefi[Sets]\label{char}~
\begin{itemize}
\item A subset $A\subset \R$ is given by its characteristic function $F_{A}:\R\di \{0,1\}$, i.e.\ we write $x\in A$ for $ F_{A}(x)=1$, for any $x\in \R$.
\item A subset $O\subset \R$ is \emph{open} in case $x\in O$ implies that there is $k\in \N$ such that $B(x, \frac{1}{2^{k}})\subset O$.
\item A subset $C\subset \R$ is \emph{closed} if the complement $\R\setminus C$ is open. 
\item A set $A\subset \R$ is \emph{enumerable} if there is a sequence of reals that includes all elements of $A$.
\item A set $A\subset \R$ is \emph{countable} if there is $Y: \R\di \N$ that is injective on $A$, i.e.\
\[
(\forall x, y\in A)( Y(x)=_{0}Y(y)\di x=_{\R}y).  
\]
\item A set $A\subset \R$ is \emph{measure zero} if for any $\eps>0$ there is a sequence of open intervals $(I_{n})_{n\in \N}$ such that $\cup_{n\in \N}I_{n}$ covers $A$ and $\eps>\sum_{n=0}^{\infty}|I_{n}|$. 
\end{itemize}
\edefi
\noindent
As discussed in Remark \ref{dichtbij}, the study of regulated functions already gives rise to open sets that 
do not come with additional representation beyond Definition~\ref{char}.  We will often assume $(\exists^{2})$ from Section \ref{lll} to guarantee that
basic objects like the unit interval are sets in the sense of Def.\ \ref{char}.  An interesting constructive enrichment of open sets from \cite{dagsamVII} is as follows.
\bdefi\label{R2D}
A set $O\subset \R$ is \emph{R2-open} if there is $Y:\R\di \R$ such that $x\in O$ implies $Y(x)>0\wedge B(x, Y(x))\subset O$.  A set $C$ is R2-closed if $\R\setminus C$ is R2-open. 
\edefi
The R2-representation is strictly weaker than the RM-representation, i.e.\ unions of basic open sets as in \cite{simpson2}*{II.5.6}.  Nonetheless, the Baire category theorem for R2-open sets is provable in $\ACAo$ by \cite{dagsamVII}*{Theorem 7.10}.

\smallskip

Secondly, we study the following weak continuity notions, many of which are well-known and hark back to the days of Baire, Darboux, Dini, Hankel, and Volterra (\cites{beren,beren2,darb, volaarde2,hankelwoot,hankelijkheid,dinipi}).  
We use `sup' and related operators in the same `virtual' or `comparative' way as in second-order RM (see e.g.\ \cite{simpson2}*{X.1}).  In this way, a formula of the form `$\sup A>a$' makes sense as shorthand for a formula in the language of all finite types, even when $\sup A$ need not exist in $\RCAo$.  
As in \cite{basket, basket2}, the definition of Baire $n$-function proceeds via (external) induction over standard $n$.
\bdefi\label{flung} For $f:[0,1]\di \R$, we have the following definitions:
\begin{itemize}
\item $f$ is \emph{upper semi-continuous} at $x_{0}\in [0,1]$ if $f(x_{0})\geq_{\R}\lim\sup_{x\di x_{0}} f(x)$,
\item $f$ is \emph{lower semi-continuous} at $x_{0}\in [0,1]$ if $f(x_{0})\leq_{\R}\lim\inf_{x\di x_{0}} f(x)$,
\item $f$ is \emph{regulated} if for every $x_{0}$ in the domain, the `left' and `right' limit $f(x_{0}-)=\lim_{x\di x_{0}-}f(x)$ and $f(x_{0}+)=\lim_{x\di x_{0}+}f(x)$ exist.  
\item $f$ is \emph{Baire 0} if it is a continuous function. 
\item $f$ is \emph{Baire $n+1$} if it is the pointwise limit of a sequence of Baire $n$ functions.
\item $f$ is \emph{effectively Baire $n$} $(n\geq 2)$ if there is a sequence $(f_{m_{1}, \dots, m_{n}})_{m_{1}, \dots, m_{n}\in \N}$ of continuous functions such that for all $x\in [0,1]$, we have 
\[\textstyle
f(x)=\lim_{m_{1}\di \infty}\lim_{m_{2}\di \infty}\dots \lim_{m_{n}\di \infty}f_{m_{1},\dots ,m_{n}}(x).
\]
\item $f$ is \emph{Baire 1$^{*}$} if\footnote{The notion of Baire 1$^{*}$ goes back to \cite{ellis} and equivalent definitions may be found in \cite{kerkje}.  
In particular, Baire 1$^{*}$ is equivalent to the Jayne-Rogers notion of \emph{piecewise continuity} from \cite{JR}.} there is a sequence of closed sets $(C_{n})_{n\in \N}$ such $[0,1]=\cup_{n\in \N}C_{n}$ and $f_{\upharpoonright C_{m}}$ is continuous for all $m\in \N$.
\item $f$ is \emph{countably continuous}\footnote{The notion of countably discontinuity, under a different name, goes back to Lusin (see \cite{novady}).} if there is a sequence of sets $(E_{n})_{n\in \N}$ such $[0,1]=\cup_{n\in \N}E_{n}$ and $f_{\upharpoonright E_{m}}$ is continuous for all $m\in \N$.
\item $f$ is \emph{quasi-continuous} at $x_{0}\in [0, 1]$ if for $ \epsilon > 0$ and any open neighbourhood $U$ of $x_{0}$, 
there is non-empty open ${ G\subset U}$ with $(\forall x\in G) (|f(x_{0})-f(x)|<\eps)$.
\item $f$ is \emph{cliquish} at $x_{0}\in [0, 1]$ if for $ \epsilon > 0$ and any open neighbourhood $U$ of $x_{0}$, 
there is a non-empty open ${ G\subset U}$ with $(\forall x, y\in G) (|f(x)-f(y)|<\eps)$.
\item $f$ is \emph{upper \(resp.\ lower\) quasi-continuous} at $x_{0}\in [0, 1]$ if for $ \epsilon > 0$ and any open neighbourhood $U$ of $x_{0}$, 
there is a non-empty open ${ G\subset U}$ with $(\forall x\in G) (f(x)< f(x_{0})+\eps)$ \(resp.\  $(\forall x\in G) (f(x)> f(x_{0})-\eps)$\).
\end{itemize}
\edefi
As to notations, a common abbreviation is `usco' and `lsco' for the first two items and `uqco' (resp.\ lqco) for the final item. 
Moreover, if a function has a certain weak continuity property at all reals in $[0,1]$ (or its intended domain), we say that the function has that property.  
Regarding the notion of `effectively Baire $n$' in Definition \ref{flung}, the latter is used, using codes for continuous functions, in second-order RM (see \cite{basket, basket2}). 
Baire himself notes in \cite{beren2}*{p.\ 69} that Baire 2 functions can be \emph{represented} by effectively Baire~2 functions.  By the results in \cite{dagsamXIV}, there is a 
significant difference between the these notions.

\smallskip

Thirdly, the following sets are often crucial in proofs relating to discontinuous functions, as can be observed in e.g.\ \cite{voordedorst}*{Thm.\ 0.36}.
\bdefi
The sets $C_{f}$ and $D_{f}$ \(if they exist\) respectively gather the points where $f:\R\di \R$ is continuous and discontinuous.
\edefi
One problem with the sets $C_{f}, D_{f}$ is that the definition of continuity involves quantifiers over $\R$.  
In general, deciding whether a given $\R\di \R$-function is continuous at a given real, is as hard as $\exists^{3}$ from Section \ref{lll}.
For these reasons, the sets $C_{f}, D_{f}$ only exist in strong systems.
A solution is discussed in Section \ref{prebairen}.

\smallskip

Fourth, we introduce some notions most of which are found already in e.g.\ the work of Volterra, Smith, and Hankel (\cites{hankelwoot, snutg, volaarde2}).
\bdefi~
\begin{itemize}
\item A set $A\subset \R$ is \emph{dense} in $B\subset \R$ if for $k\in \N,b\in B$, there is $a\in A$ with $|a-b|<\frac{1}{2^{k}}$.
\item A function $f:\R\di \R$ is \emph{pointwise discontinuous} if for any $x\in \R, k\in \N$ there is $y\in B(x, \frac{1}{2^{k}})$ such that $f$ is continuous at $y$.
\item A set $A\subset \R$ is \emph{nowhere dense} in $B\subset \R$ if $A$ is not dense in any open sub-interval of $B$.  
\item A function $f:[0,1]\di \R$ is \emph{simply continuous} \(sico\) if for any open $G\subset \R$, the set $f^{-1}(G)$ is the union of an open and a nowhere dense set. 
\end{itemize}
\edefi
\noindent
Fifth, we also need the notion of `intermediate value property', also called the `Darboux property' in light of Darboux's work in \cite{darb}.
\bdefi[Darboux property] Let $f:[0,1]\di \R$ be given. 
\begin{itemize}
\item A real $y\in \R$ is a left \(resp.\ right\) \emph{cluster value} of $f$ at $x\in [0,1]$ if there is $(x_{n})_{n\in \N}$ such that $y=\lim_{n\di \infty} f(x_{n})$ and $x=\lim_{n\di \infty}x_{n}$ and $(\forall n\in \N)(x_{n}\leq x)$ \(resp.\ $(\forall n\in \N)(x_{n}\geq x)$\).  
\item A point $x\in [0,1]$ is a \emph{Darboux point} of $f:[0,1]\di \R$ if for any $\delta>0$ and any left \(resp.\ right\) cluster value $y$ of $f$ at $x$ and $z\in \R$ strictly between $y$ and $f(x)$, 
there is $w\in (x-\delta, x)$ \(resp.\ $w\in ( x, x+\delta)$\) such that $f(w)=y$.   
\end{itemize}
\edefi
\noindent
By definition, a point of continuity is also a Darboux point, but not vice versa.  

\smallskip

Finally, the following remark is meant to express that closed sets as in Definition~\ref{char}, i.e.\ without any additional representation (say provable in $\Z_{2}^{\omega}$), are readily found in the wild.  
This even holds for the `weak' representation from Def.\ \ref{R2D}.
\begin{rem}\label{dichtbij}\rm
First of all, fix a regulated function $f:[0,1]\di \R$ and consider the set $D_{k}$ as in \eqref{howie}, definable using $\exists^{2}$ and such that $D_{f}=\cup_{k\in \N}D_{k}$. 
\be\label{howie}\textstyle
D_{k}:=\{x\in [0,1]:  |f(x)-f(x+)|>\frac{1}{2^{k}}\vee |f(x)-f(x-)|>\frac{1}{2^{k}}\}.
\ee
This set is central to many proofs involving regulated functions (see e.g.\ \cite{voordedorst}*{Thm.~0.36}).  
Now, that $D_{k}$ is finite follows by a standard\footnote{If $A_{n}$ were infinite, the Bolzano-Weierstrass theorem implies the existence of a limit point $y\in [0,1]$ for $A_{n}$.  One readily shows that $f(y+)$ or $f(y-)$ does not exist, a contradiction as $f$ is assumed to be regulated.  This argument is readily formalised in $\ACAo+\QFAC^{0,1}$.\label{fkluk}} compactness argument.     
Hence, working in $\Z_{2}^{\Omega}$ from Section \ref{lll}, one readily proves the following:
\begin{center}
\emph{there is a sequence $(C_{k})_{k\in \N}$ of RM-codes for the sequence of closed sets $(D_{k})_{k\in \N}$,}
\end{center}
\begin{center}
\emph{there is a sequence $(Y_{k})_{k\in \N}$ of R2-representations for the closed sets $(D_{k})_{k\in \N}$.}
\end{center}
However, $\Z_{2}^{\omega}$ \emph{cannot}\footnote{The Baire category theorem for RM-open and R2-open sets is provable in $\ACAo$ by \cite{simpson2}*{II.5.8} and \cite{dagsamVII}*{Theorem 7.10}.  
Hence, in case $D_{k}$ as in \eqref{howie} have RM-codes or R2-representations, the associated Baire category theorem implies that $D_{f}=\cup_{k\in \N}D_{k}$ is not $[0,1]$, i.e.\ $C_{f}$ is non-empty.  
By \cite{samBIG}*{Theorem 3.7}, we obtain $\NIN_{[0,1]}$ which is not provable in $\Z_{2}^{\omega}$ by \cite{dagsamX}*{Theorem 3.1}. } prove either of the centred coding statements about \eqref{howie}.
Since regulated functions are usco and cliquish (by definition, say in $\RCAo$), this observation applies to the topic of this paper as well.  We can obtain the same results for functions of bounded variation, following  \cite{samBIG}*{\S3.4}.
\end{rem}

\subsubsection{On oscillation functions}\label{prebairen}
In this section, we introduce \emph{oscillation functions} and establish some required properties.
We have previously studied usco, Baire 1 and cliquish functions in Kleene's computability theory using oscillation functions (see \cite{samcsl23}). 
As will become clear, such functions are generally necessary for the RM-study in Section \ref{mainn} and \ref{related}.  

\smallskip

First of all, the study of regulated functions in \cites{dagsamXI, dagsamXII, dagsamXIII} is 
really only possible thanks to the associated left- and right limits (see Definition \ref{flung}) \emph{and} the fact that the latter are computable in $\exists^{2}$.  
Indeed, for regulated $f:\R\di \R$, the formula 
\be\label{figo}\tag{\textup{\textsf{C}}}
\text{\emph{ $f$ is continuous at a given real $x\in \R$}}
\ee
involves quantifiers over $\R$ but is equivalent to the \emph{arithmetical} formula $f(x+)=f(x)=f(x-)$.  
In this light, we can define the set $D_{f}$ of discontinuity points of $f$ -using only $\exists^{2}$- and proceed with the usual (textbook) proofs.  
An analogous approach, namely the study of usco, Baire 1, and cliquish functions in Kleene's computability theory, was used in \cite{samcsl23}.  
To this end, we used \emph{oscillation functions} as in Definition \ref{oscfn}.  We note that Riemann and Hankel already considered the notion of oscillation in the context of Riemann integration (\cites{hankelwoot, rieal}).  
\bdefi[Oscillation functions]\label{oscfn}
For any $f:\R\di \R$, the associated \emph{oscillation functions} are defined as follows: $\osc_{f}([a,b]):= \sup _{{x\in [a,b]}}f(x)-\inf _{{x\in [a,b]}}f(x)$ and $\osc_{f}(x):=\lim _{k \di \infty }\osc_{f}(B(x, \frac{1}{2^{k}}) ).$
\edefi
 We  stress that $\osc_{f}:\R\di \R$ is \textbf{only}\footnote{To be absolutely clear, the notation `$\osc_{f}$' and the appearance of $f$ therein in particular, is purely symbolic and lambda abstraction involving `$\lambda f$' is expressly forbidden.} 
 a third-order function, as clearly indicated by its type.   On a related technical note, while the suprema, infima, and limits in Definition~\ref{oscfn} do not always exist in weak systems, formulas like $\osc_{f}(x)>y$ \emph{always} make sense as shorthand 
 for the standard definition of the suprema, infima, and limits involved; this `virtual' or `comparative' meaning is part and parcel of (second-order) RM in light of \cite{simpson2}*{X.1}.

\smallskip

Now, our main interest in Definition \ref{oscfn} is that \eqref{figo} is equivalent to the \emph{arithmetical} formula $\osc_{f}(x)=0$, assuming the latter function is given.  
Hence, in the presence of $\osc_{f}:\R\di \R$ and $\exists^{2}$, we can define $D_{f}$ and proceed with the usual (textbook) proofs, which is the approach we took in \cite{samcsl23}.   
Below, we also show that we can avoid the use of oscillation functions for usco functions. 

\smallskip

Secondly, we sketch the connection between usco (or: cliquish) functions and the Baire category theorem, in both directions, to further motivate our use of oscillation functions.  
In one direction, fix a usco function $f:[0,1]\di \R$ and its oscillation function $\osc_{f}:\R\di \R$.  A standard textbook technique is to decompose the set $D_{f}=\{ x\in [0,1]: \osc_{f}(x)>0  \}$ as the union of the closed sets
\be\label{kokn}\textstyle
D_{k}:=\{ x\in [0,1]: \osc_{f}(x)\geq \frac{1}{2^{k}}  \} \textup{  for all $k\in \N$.}
\ee
The complement $O_{n}:= [0,1]\setminus D_{k}$ can be shown to be open and dense, as required for the antecedent of the Baire category theorem, i.e.\ $C_{f}$ is non-empty as a result.  
This connection also goes in the other direction as follows: fix a sequence of dense and open sets $(O_{n})_{n\in \N}$ in the unit interval, define $X_{n}:= [0,1]\setminus O_{n}$ and consider 
the following function $h:[0,1]\di \R$ defined using $\mu^{2}$:
\be\label{mopi}\tag{\textsf{\textup{H}}}
h(x):=
\begin{cases}
0 & x\not \in \cup_{m\in \N}X_{m} \\
\frac{1}{2^{n+1}} &  x\in X_{n} \textup{ and $n$ is the least such number}
\end{cases}.
\ee
The function $h$ may be found in the literature (\cite{myerson}*{p.\ 238}) and is usco by Theorem~\ref{flap}.  
Moreover, for $x\in C_{h}$ we also have $x\in \cap_{n\in \N}O_{n}$, as required for the consequent of the Baire category theorem.
Thus, the Baire category theorem is intimately connected to continuity properties of usco functions (in both directions).  
One can similarly obtain the same connection for e.g.\ Baire $1^{*}$ and cliquish functions,  \emph{assuming} we have access to \eqref{kokn}, which is why
we assume $\osc_{f}:\R\di \R$ to be given in general.    

\smallskip

Thirdly, much to our own surprise, the `counterexample' function from \eqref{mopi} has nice properties \emph{that are provable in weak systems}.  
\begin{thm}[$\ACAo$]\label{fronkn}
Let $(X_{n})_{n\in \N}$ be an increasing sequence of closed nowhere dense sets.  Then $h:[0,1]\di \R$ from \eqref{mopi} is its own oscillation function. 
\end{thm}
\begin{proof}
Consider $h:[0,1]\di \R$ as in \eqref{mopi} where $(X_{n})_{n\in \N}$ is an increasing sequence of closed nowhere dense sets.  
We will show that $h(x)=\osc_{h}(x)$ for all $x\in [0,1]$, i.e.\ $h$ is its own oscillation function. 
To this end, we proceed by case distinction. 
\begin{itemize}
\item In case $h(x_{0})=0$ for some $x_{0}\in [0,1]$, then $x_{0}\in \cap_{n\in \N}Y_{n}$ where $Y_{n}:= [0,1]\setminus X_{n}$ is open.  Hence, for any $m\in \N$, there is $N\in \N$ such that $B(x_{0}, \frac{1}{2^{N}})\subset Y_{m}$, as the latter is open. By the definition of $\osc_{h}$, we have $\osc_{h}(x_{0})<\frac{1}{2^{m}}$ for all $m\in \N$, i.e.\ $\osc_{h}(x_{0})=h(x_{0})=0$.  
\item In case $\osc_{h}(x_{0})=0$ for some $x_{0}\in [0,1]$, we must have $x_{0}\not \in  \cup_{n\in \N}X_{n}$ and hence $h(x_{0})=0$ by definition.  
Indeed, if $x_{0}\in X_{n_{0}}$, then $\osc_{h}(x_{0})\geq \frac{1}{2^{n_{0}+1}}$ because $\inf_{x\in B(x_{0}, \frac{1}{2^{k}})}h(x)=0$  due to $Y_{m}$ being dense in $[0,1]$ for any $m\in \N$, 
while of course $\sup_{x\in B(x_{0}, \frac{1}{2^{k}})}h(x)\geq h(x_{0})\geq \frac{1}{2^{n_{0}+1}}$.   
\item In case $h(x_{0})=\frac{1}{2^{n_{0}+1}}$ for some $x_{0}\in [0,1]$, suppose $\osc_{h}(x_{0})\ne \frac{1}{2^{n_{0}+1}}$.   
Since by definition (and the previous item) $\osc_{h}(x_{0})\geq h(x_{0})=\frac{1}{2^{n_{0}+1}}$, we have $\osc_{h}(x_{0})>\frac{1}{2^{n_{0}+1}}$, implying $\osc_{h}(x_{0})\geq\frac{1}{2^{n_{0}}}$ and $n_{0}>0$.  Now, if $x_{0}\in Y_{n_{0}-1}$, then $B(x_{0}, \frac{1}{2^{N}})\subset Y_{n_{0}-1}$ for $N$ large enough, as $Y_{n_{0}-1}$ is open; by definition, the latter inclusion implies that $\osc_{h}(x_{0})\leq \frac{1}{2^{n_{0}}}$, a contradiction.  However, $x_{0}\not \in Y_{n_{0}-1}$  also leads to a contradiction as then $h(x_{0})\geq \frac{1}{2^{n_{0}}}$.  In conclusion, we have $h(x_{0})=\frac{1}{2^{n_{0}+1}}=\osc_{h}(x_{0})$. 
\item In case $\osc_{h}(x_{0})>0$ for some $x_{0}\in [0,1]$, suppose $h(x_{0})\ne \osc_{h}(x_{0})$.   
By definition (and the first item) we have $\osc_{h}(x_{0})\geq h(x_{0})>0$, implying $\osc_{h}(x_{0})>h(x_{0})=\frac{1}{2^{n_{0}+1}}$ for some $n_{0}\in \N$.  
In turn, we must have $\osc_{h}(x_{0})\geq \frac{1}{2^{n_{0}}}$ and $n_{0}>0$.
Now, if $x_{0}\in Y_{n_{0}}$, then $B(x_{0}, \frac{1}{2^{N}})\subset Y_{n_{0}}$ for $N$ large enough, as $Y_{n_{0}}$ is open; by definition, the latter inclusion implies that $\osc_{h}(x)\leq \frac{1}{2^{n_{0}+1}}$, a contradiction.  However, $x_{0}\not \in Y_{n_{0}}$, also leads to a contradiction as then $h(x_{0})\geq\frac{1}{2^{n_{0}}}$.  
Since both cases lead to contradiction, we have $h(x_{0})=\osc_{h}(x_{0})$. 
\end{itemize}
We thus have $h(x)=\osc_{h}(x)$ for all $x\in [0,1]$, as required.  
\end{proof}
As it happens, functions that are their own oscillation function have been studied in the mathematical literature (see e.g.\ \cite{kosten}). 
The previous theorem is surprising: computing the oscillation function of arbitrary functions is as hard as $\exists^{3}$.  

\smallskip

Fourth, we establish the well-known fact that oscillation functions are usco.
\begin{thm}[$\ACAo$]\label{flappie}
For any $f:\R\di \R$ such that the oscillation function $\osc_{f}:\R\di \R$ exists, the latter is usco.  
\end{thm}
\begin{proof}
Fix $x_{0}\in \R$ and consider $y\in \R$ such that $y>\osc_{f}(x_{0})$.  
Say $y>\osc_{f}(x_{0})+\eps$ for $\eps>0$ and let $n_{0}$ be such that for all $n\geq n_{0}$, we have $|\osc_{f}(B(x_{0}, \frac{1}{2^{n}}))-\osc_{f}(x_{0})|<\eps/2$.  
Now pick $z\in \R$ and $m_{0}\in \N$ such that $B(z, \frac{1}{2^{m_{0}}})\subset B(x_{0}, \frac{1}{2^{n_{0}}})$.  
By definition, we have $\osc_{f}([a,b])\leq \osc_{f}([c, d])$ in case $[a,b]\subseteq [c,d]$. 
Hence, we have $\osc_{f}(z)\leq \osc_{f}(B(z, \frac{1}{2^{m_{0}}}))\leq \osc_{f}(B(x_{0}, \frac{1}{2^{n_{0}}}))\leq \osc_{f}(x_{0})+\eps/2<y$, i.e.\ $f$ is usco at $x_{0}$ and we are done.
\end{proof}
Fifth, we provide a direct proof of some properties of equation \eqref{mopi}.
\begin{thm}[$\ACAo$]\label{flap}
Let $(X_{n})_{n\in \N}$ be an increasing sequence of closed sets.
The function $h:[0,1]\di \R$ from \eqref{mopi} is usco and cliquish.  
\end{thm}
\begin{proof}
First of all, fix $x_{0}\in [0,1]$ and note that if $h(x_{0})=0$, we have $x_{0}\in Y_{n}:=[0,1]\setminus X_{n}$ for all $n\in \N$.  
These sets are open, i.e.\ for each $n\in \N$ there is $N\in \N$ with $B(x_{0}, \frac{1}{2^{N}})\subset Y_{n}$.  Then $h(x)< \frac{1}{2^{n+1}}$ in the latter ball, i.e.\ the definition of usco is satisfied. 
The definition of usco is trivially satisfied if $h(x)=\frac{1}{2}$ as the latter is the supremum of $h$ on $[0,1]$.  
In case $h(x_{0})=\frac{1}{2^{n_{0}+1}}$ for $n_{0}>0$, then we have $(\exists N\in \N)(\forall y\in B(x_{0}  , \frac{1}{2^{N}}   )(y\not\in X_{n_{0}-1})$.  
Indeed, if $(\forall N\in \N)(\exists y\in B(x_{0}  , \frac{1}{2^{N}}   )(y\in X_{n_{0}-1})$, then $y\in X_{n_{0}-1}$ as the latter set is closed, contradicting the leastness of $n_{0}$. 
For $N_{0}\in \N$ such that $(\forall y\in B(x_{0}  , \frac{1}{2^{N_{0}}}   )(y\not\in X_{n_{0}-1})$, we have $h(y)\leq \frac{1}{2^{n_{0}+1}}$ in this ball. 

\smallskip

Secondly, $h$ from \eqref{mopi} is also cliquish.  To see this, fix $\eps>0$ and $x_{0}\in [0,1]$ and take large enough $n_{0}\in \N$ such that $\frac{1}{2^{n_{0}}}<\eps$.  
Since $Y_{n_{0}}$ is dense, we can find $x_{1}\in Y_{n_{0}}$ arbitrarily close to $x_{0}$.   However, since $Y_{n_{0}}$ is open, we have $B(x_{1}, \frac{1}{2^{N_{1}}})\subset Y_{n_{0}}$ for some $N_{1}\in \N$. 
Since $h$ is at most $\frac{1}{2^{n_{0}+1}}$ on this ball, we observe that the definition of cliquish function is satisfied.   
\end{proof}
The previous theorem is illustrative as follows: in case each $X_{n}$ is a `fat' Cantor set of measure at least $1-\frac{1}{2^{n}}$, then $h$ from \eqref{mopi} is a usco function that is continuous on a measure zero set, namely the complement of $\cup_{n\in \N}X_{n}$.  These facts can be established in $\ACAo$ in light of \cite{simpson2}*{p.\ 129}.  

\smallskip

Finally, the computational equivalences in \cite{samcsl23} deal with usco, Baire 1, and cliquish functions, while the RM-equivalences in this paper are only about usco and cliquish functions.
The deeper reason for the absence of `Baire 1' in our equivalences is unveiled by Theorems \ref{toblo}, \ref{flang}, and \ref{qvl}: even for the smaller class of usco functions, the extra assumption that the latter
are also Baire 1, makes e.g.\ the semi-continuity lemma provable in much weaker systems.

\section{The Baire category theorem}\label{mainn}
In this section, we obtain the equivalences sketched in Section \ref{vower} involving the Baire category theorem and theorems about semi-continuous and cliquish functions (see Section \ref{cdef} for the latter notions).  
We recall that the Baire category theorem has been studied in second-order RM (see e.g.\ \cite{trevor,mietje}); 
the use of codes for open sets means the theorem is provable in the base theory $\RCA_{0}$ or similarly weak system.  
To avoid issues of the representation of real numbers and sets thereof, we will always assume $\mu^{2}$ or $\exists^{2}$ from Section~\ref{lll}.
We shall often make use of the preliminary results from Section \ref{prebairen}.  

\subsection{Introduction}
In this section, we introduce the Baire category theorem and describe the results to be obtained in the below sections.  

\smallskip

First of all, we shall study the Baire category theorem formulated as follows.  
\begin{princ}[$\BCT_{[0,1]}$]\label{BCTI}
If $ (O_n)_{n \in \N}$ is a decreasing sequence of dense open sets of reals in $[0,1]$, then $ \bigcap_{n \in\N } O_n$ is non-empty.
\end{princ} 
\noindent
We assume that $O_{n+1}\subseteq O_{n}$ for all $n\in \N$ to avoid the use of induction to prove that a finite intersection of open and dense sets is again open and dense. 

\smallskip

Secondly, $\BCT_{[0,1]}$ is hard to prove in terms of (conventional) comprehension by \cite{dagsamVII}*{Theorem 6.16}, while the following theorem provides a slick proof of this fact.  
\begin{thm}\label{tonko}
The system  $\Z_{2}^{\omega}+\IND_{0}$ cannot prove $\BCT_{[0,1]}$.
\end{thm}
\begin{proof}
Let $Y:[0,1]\di \N$ be an injection and define $C_{n}:=\{x\in [0,1]: Y(x)\leq n\}$.   Use $\IND_{0}$ to prove that this set is closed and nowhere dense for each $n\in \N$.  
By assumption, $[0,1]=\cup_{n\in \N}C_{n}$, i.e.\ we have established  $\BCT_{[0,1]}\di \NIN_{[0,1]}$ via contraposition; the latter is not provable in $\Z_{2}^{\omega}+\IND_{0}$ by \cite{dagsamXI}*{Theorem 2.16}. 
\end{proof}
By contrast, the restriction of $\BCT_{[0,1]}$ to R2-open sets from Section \ref{cdef} can be proved in $\ACAo$ by \cite{dagsamVII}*{Thm 7.10}.
Basically, one follows the constructive proof (see e.g.\ \cite{bish1}) while avoiding the Axiom of (countable) Choice thanks to the R2-representation.  
This begs the question which other constructive enrichments make $\BCT_{[0,1]}$ (and equivalent theorems) provable in weak(er) systems.

\smallskip

We explore this `constructivisation' theme in Theorems \ref{clockn}, \ref{toblo}, and \ref{plonkook} as follows: the latter theorems show that replacing `cliquish' by the closely related notion `quasi-continuous' (see Remark \ref{donola}) results in theorems provable in $\ACAo$.  The same holds for the addition of `Baire 1' to theorems about usco functions; there is no contradiction here as the (classically true) statement
\begin{center}
\emph{a bounded usco function on the unit interval is Baire 1} 
\end{center}
already implies (stronger principles than) $\BCT_{[0,1]}$ by Corollary~\ref{lopi}.  Nonetheless, cliquishness and quasi-continuity (and the same for Baire 1 and usco) are closely related notions (see again Remark \ref{donola}), yet there is a great divide -or abyss- between their continuity properties.  

\smallskip

Thirdly, we shall obtain the following RM-results in the below sections.  
\begin{itemize}
\item We connect $\BCT_{[0,1]}$ to the \emph{semi-continuity lemma} in Section \ref{SCL}.
\item We generalise the results in Section \ref{SCL} to \emph{Baire $1^{*}$} in Section \ref{bong}.
\item We connect $\BCT_{[0,1]}$ to the \emph{uniform boundedness principle} in Section \ref{UBP}.
\item We connect $\BCT_{[0,1]}$ to properties of \emph{cliquish} functions in Section \ref{cliq}.
\item We connect $\BCT_{[0,1]}$ to properties of \emph{fragmented} functions in Section \ref{1227}.
\item We connect $\BCT_{[0,1]}$ to Volterra's early work (Section \ref{vintro}) in Section \ref{vork}.
\end{itemize}
As discussed in Remark \ref{donola}, the class of cliquish functions is quite large/wild and is identical to the class of 
pointwise discontinuous function, i.e.\ we study Volterra's theorem in full generality.  We note that `Baire 1' and `fragmented' are equivalent on $\R$ in 
strong enough systems (see Section \ref{1227}), but not in $\Z_{2}^{\omega}$ by Corollary~\ref{tokkiecor}. 

\smallskip

Finally, we have studied the RM of the \emph{gauge integral} in \cites{dagsamIII, dagsamVI}.
This integral generalises the Lebesgue integral but remains conceptually close to the Riemann integral by -intuitively speaking- replacing the real `$\delta>0$'  
by a \emph{gauge function} `$\delta:\R\di \R^{+}$' in the usual epsilon-delta definition (\cite{bartle}).  This gauge function can be taken to be usco (\cite{gepeperd}) in very general cases, i.e.\ 
the RM of usco functions in this paper may well shed new light on the RM of the gauge integral.

\subsection{The semi-continuity lemma}\label{SCL}
We establish the equivalence between $\BCT_{[0,1]}$ and the \emph{semi-continuity lemma}; 
the latter expresses that the continuity points of a semi-continuous functions are dense, and is therefore not provable in $\Z_{2}^{\omega}$. 
We also identify (slight) variations that are provable in weaker systems like $\ACAo$.  

\smallskip

First of all, we have the following theorem pertaining to usco functions; note that all items are immediately generalised from `one point of continuity' to `a dense set of points of continuity'.
Regarding item \eqref{bctv} and \eqref{bct6}, there are semi-continuous functions that are not countably continuous (\cite{novady}) or not bounded\footnote{Define $f:\R\di \R$ as: $f(x)=\frac{-1}{x}$ if $x>0$ and $f(x)=1$ if $x\leq 0$, which is usco but not bounded below.  By \cite{dagsamXIV}*{Theorem 2.9}, usco functions are bounded above in $\RCAo+\QFAC^{0,1}+\WKL$.} below.
\begin{thm}[$\ACAo$]\label{clockn} The following are equivalent. 
\begin{enumerate}
 \renewcommand{\theenumi}{\alph{enumi}}
\item The Baire category theorem as in $\BCT_{[0,1]}$.\label{bcti}
\item For usco $f:[0,1]\di \R$, there is a point $x\in [0,1]$ where $f$ is continuous.\label{bctii}
\item For usco $f:[0,1]\di \R$ which has an oscillation function $\osc_{f}:[0,1]\di \R$, there is a point $x\in [0,1]$ where $f$ is continuous. \label{bctiii}
\item For usco $f:[0,1]\di \R$ which is its own oscillation function, there is a point $x\in [0,1]$ where $f$ is continuous. \label{bctiv}
\item For usco and countably continuous $f:[0,1]\di \R$, there is a point $x\in [0,1]$ where $f$ is continuous.\label{bctv}
\item Any usco $f:[0,1]\di \R$ is bounded below on some interval, i.e.\ there exist $q\in \Q$ and $c, d\in [0,1]$ such that $(\forall y\in (c, d))(f(y)\geq q)$.\label{bct6}
\item For usco $f:[0,1]\di \R^{+}$, there are $c, d\in [0,1]$ with $0<\inf_{x\in [c, d]}f(x)$.\label{bct7}
\end{enumerate}
\end{thm}
\begin{proof}
First of all, assume item \eqref{bctiv} and fix a decreasing sequence of dense open sets $(O_{n})_{n\in \N}$.  
Define the nowhere dense closed set $X_{n}:= [0,1]\setminus O_{n}$ and consider $h:[0,1]\di \R$ as in \eqref{mopi}.  
Following Theorems \ref{fronkn} and \ref{flap}, item \eqref{bctiv} yields a point of continuity $y\in [0,1]$ of $h$. 
If $\frac{1}{2^{m_{0}+1}}=h(y)>0$, then by definition $y\in X_{m_{0}}$ where $m_{0}\in \N$ is the least such number.  
Since $h$ is continuous at $y$, there is $N\in \N$ such that $z\in B(y, \frac{1}{2^{N}})$ implies $|h(z)-h(y)|<\frac{1}{2^{m_{0}+2}}$.
However, $O_{m_{0}+2}$ is dense in $[0,1]$, implying there is $z_{0}\in B(y, \frac{1}{2^{N}})\cap O_{m_{0}+2}$.  
Since $h(z_{0})\leq\frac{1}{2^{m_{0}+3}}$ by definition, we obtain a contradiction.  Hence, we must have $h(y)=0$ and the latter implies $y\in \cap_{n\in \N}O_{n}$ by definition, as required for $\BCT_{[0,1]}$.

\smallskip

Secondly, for item \eqref{bctii}, we first formulate a proof (say in $\Z_{2}^{\Omega}$) and then show how this proof goes through in $\ACAo+\BCT_{[0,1]}$.
Thus, fix lsco $f:[0,1]\di \R$ and define $C_{q}:=\{x\in [0,1]: f(x)\leq q\}$, which is closed by definition. 
Define $F_{q}$ as $C_{q}\setminus \textsf{int}(C_{q})$, where the open set $\textsf{int}(X)$ is the interior of $X$, i.e.\ those $x\in X$ such that there is $N\in \N$ with $B(x, \frac{1}{2^{N}})\subset X$. 
By definition, $F_{q}$ is closed and is nowhere dense, i.e.\ the Baire category theorem shows that there is $y\in \big([0,1 ]\setminus \cup_{q\in \Q}F_{q}\big)$.  
The Baire category theorem is provable in $\Z_{2}^{\Omega}$ by \cite{dagsamVII}*{\S6}.

\smallskip

We now show that for each $x\in [0,1]$ where $f$ is discontinuous, there is $q\in \Q$ such that $x\in F_{q}$, establishing that $f$ is continuous at the aforementioned point $y$, as required. 
Thus, let $f$ be discontinuous at $x_{0}\in [0,1]$ and note that $f$ cannot be usco at $x_{0}\in [0,1]$, i.e.\ we have 
\be\label{ting}\textstyle
(\exists l\in \N)(\forall N\in \N)(\exists z\in B(x_{0}, \frac{1}{2^{N}}))(f(z)\geq f(x_{0})+\frac{1}{2^{l}}).
\ee
Let $l_{0}$ be as in \eqref{ting} and consider $q_{0}\in\Q$ such that $f(x_{0})<q_{0}< f(x_{0})+\frac{1}{2^{l_{0}}}$. 
By definition, $x_{0}\in F_{q_{0}}$, as required.  

\smallskip

We now show how the above proof goes through in $\ACAo+\BCT_{[0,1]}$.  The (only) two obstacles are as follows:
\begin{itemize}
\item the definition of $\textsf{int}(X)$ involves quantification over $\R$, i.e.\ the former set cannot be defined in $\ACAo$.
\item the formula \eqref{ting} involves quantification over $\R$, i.e.\ the numbers $l_{0}$ and $q_{0}$ cannot be defined in terms of $x_{0}$.  
\end{itemize}
Both problems have (about) the same solution, as follows.  
For the second item, since $f$ is lsco, \eqref{ting} is equivalent to 
\be\label{ting2}\textstyle
(\exists l\in \N)(\forall N\in \N)\underline{(\exists r\in B(x, \frac{1}{2^{N}})\cap \Q)}(f(r)\geq f(x_{0})+\frac{1}{2^{l}}),
\ee
where the underlined quantifier is crucial.  Since \eqref{ting2} is decidable using $\exists^{2}$, the functional $\mu^{2}$ computes $l_{1}\in \N$, the least $l\in \N$ satisfying \eqref{ting2}; clearly, $l=l_{1}+1$ then witnesses \eqref{ting}.  Hence, $q_{0}$ from the previous paragraph is definable in terms of $x_{0}$ (and $f$ and $\mu^{2}$).  

\smallskip

For the first item, i.e.\ relating to the interior of sets, the following formula expresses that $x\in F_{q}= (C_{q}\setminus \textsf{int}(C_{q})$, by definition:
\be\label{ling}\textstyle
f(x)\leq q \wedge (\forall N\in \N)(\exists z\in B(x, \frac{1}{2^{N}}))( f(z)>q  ).
\ee
Similar to the second item and \eqref{ting}, since $f$ is lsco, \eqref{ling} is equivalent to 
\be\label{ling2}\textstyle
f(x)\leq q \wedge (\forall N\in \N)\underline{(\exists r\in B(x, \frac{1}{2^{N}})\cap \Q)}( f(r)>q  ),
\ee
where again the underlined quantifier is crucial.  Hence, we can define $F_{q}$ using $\exists^{2}$ and \eqref{ling2}.  
Using an enumeration of $\Q$, one readily obtains an increasing sequence of closed nowhere dense sets.  
The rest of the proof now goes through.  Since for lsco $f:[0,1]\di \R$, the function $-f$ is usco, this item is finished.

\smallskip

For item \eqref{bctv}, it is clear that $h$ from \eqref{mopi} is countably continuous as it is constant on each $X_{n}$ and the complement of $\cup_{n\in \N}X_{n}$.

\smallskip

To prove item \eqref{bct6} from $\BCT_{[0,1]}$, fix usco $f:[0,1]\di \R$ and note that $C_{n}:= \{ x\in [0,1]: f(x)\geq -n\}$ is closed and satisfies $[0,1]=\cup_{n\in \N}C_{n}$.  
By $\BCT_{[0,1]}$, at least one $C_{n}$ is not nowhere dense, as required for item \eqref{bct6}.  Now assume the latter and let $(X_{n})_{n\in \N}$ be an increasing sequence of closed sets such that $[0,1]=\cup_{n\in \N}X_{n}$.
Define $\tilde{h}:[0,1]\di \R$ as follows: $\tilde{h}(x)=-n$ in case $x\in X_{n}$ and $n\in \N$ is the least such number.  
Similar to $h:[0,1]\di \R$ from \eqref{mopi}, the function $\tilde{h}$ is usco.   Item \eqref{bct6} provides $N\in \N$ and $c, d\in [0,1]$ such that $\tilde{h}(y)\geq -N$ for $y\in (c, d)$. 
By definition, we have $(c, d)\subset X_{N}$, i.e.\ the latter is not nowhere dense.   Item \eqref{bct7} follows in the same way using $h$ from \eqref{mopi} and $C_{n}:=\{x\in [0,1]: f(x)\geq \frac{1}{2^{n}} \}$.
\end{proof}
Secondly, proofs in real analysis often use $C_{f}$ and $D_{f}$, the set of continuity and discontinuity points, as their starting points.  
The definition of continuity involves a quantifier over $\R$, i.e.\ these sets cannot be defined in general in e.g.\ $\Z_{2}^{\omega}$. 
In light of the proof of Theorem \ref{clockn}, we can define these sets for usco functions as follows.  
\begin{thm}[$\ACAo$]\label{plonk}
For usco $f:[0,1]\di \R$, the sets $C_{f}$ and $D_{f}$ exist.  
\end{thm}
\begin{proof}
First of all, it is a matter of definitions to show the equivalence between:
\begin{itemize}
\item $g:\R\di \R$ is continuous at $x\in \R$,
\item $g:\R\di \R$ is usco and lsco at $x\in \R$.
\end{itemize}
Secondly, for lsco $f:[0,1]\di \R$, `$f$ is discontinuous at $x\in [0,1]$' is equivalent to 
\be\label{dumbel}\textstyle
(\exists l\in \N)(\forall k\in \N){(\exists y\in B(x, \frac{1}{2^{k}})}(f(y)\geq f(x)+\frac{1}{2^{l}} ),
\ee
which expresses that $f$ is not usco at $x\in [0,1]$.  Now, \eqref{dumbel} is equivalent to 
\be\label{dumbel1}\textstyle
(\exists l\in \N)(\forall k\in \N)\underline{(\exists r\in B(x, \frac{1}{2^{k}})\cap \Q)}(f(r)\geq f(x)+\frac{1}{2^{l}} ),
\ee
where in particular the underlined quantifier in \eqref{dumbel1} has rational range due to $f$ being lsco.  
Since \eqref{dumbel1} is arithmetical, $\exists^{2}$ allows us to define $D_{f}$, and hence $C_{f}$.  
\end{proof}
An important observation is that Theorem \ref{plonk} only states that the set $C_{f}$ \emph{exists} but does \emph{not} claim it to be non-empty.  
The latter claim can only\footnote{In fact, it is consistent with $\Z_{2}^{\omega}+\QFAC^{0,1}$ that there are bounded usco functions on the unit interval that are everywhere discontinuous (see \cite{dagsamXIV}*{\S3}).} be made in $\ACAo$ if we have additional information about the function $f$, e.g.\ as in Theorem \ref{toblo}.  

\smallskip

Thirdly, Theorem \ref{clockn} has a number of fairly easy generalisations. 
There is also degree of robustness: we may similarly replace `measure zero' by `measure at most $a\in (0,1)$'.
\begin{thm}\label{deng}
Theorem \ref{clockn} still goes through with the following restrictions.
\begin{itemize} 
\item Replace `continuous' in the consequent of items \eqref{bctii}-\eqref{bctiv} by:
\begin{center}
 regulated, \(lower\) quasi-continuituous, cadlag, Darboux, or lsco.
\end{center}
 \item Restrict `usco' to `usco and such that $C_{f}$ has measure zero'. 
 \item Restrict $\BCT_{[0,1]}$ to open sets such that $\cap_{n\in \N}O_{n}$ has measure zero. 
\end{itemize}
\end{thm}
\begin{proof}
For the first item, we only need to derive $\BCT_{[0,1]}$ from items items \eqref{bctii}-\eqref{bctiv} in the theorem with `continuity' weakened to the notions from the corollary. 
This is a straightforward verification and we only establish the corollary for the weakest notion, namely lower quasi-continuity.  
Thus, assume item \eqref{bctii} with `lqco' in the consequent and consider the real $y\in [0,1]$ as in the first paragraph of the proof of the theorem.
If $\frac{1}{2^{m_{0}+1}}=h(y)>0$, then by definition $y\in X_{m_{0}}$ where $m_{0}\in \N$ is the least such number.  
Since $h$ is lqso at $y$, we have that for $\eps=\frac{1}{2^{m_{0}+2}}$ and any $N\in \N$, there is $(a, b)\subset B(y, \frac{1}{2^{N}})$ such that $z\in (a, b)$ implies $h(z)>h(y)-\frac{1}{2^{m_{0}+2}}=\frac{1}{2^{m_{0}+2}}$.
However, $O_{m_{0}+4}$ is dense in $[0,1]$, implying there is $z_{0}\in (a, b)\cap O_{m_{0}+4}$.  
Since $h(z_{0})\leq\frac{1}{2^{m_{0}+5}}$, we obtain a contradiction.  Hence, we must have $h(y)=0$ and the latter implies $y\in \cap_{n\in \N}O_{n}$ by definition, as required for $\BCT_{[0,1]}$.

\smallskip

For the second item, consider the first paragraph of the proof of Theorem \ref{clockn}.  Assuming $[0,1]=\cup_{n\in \N}X_{n}$, the usco function $h$ from \eqref{mopi} is such that $C_{f}=\emptyset$, i.e.\ definitely measure zero.  
Now proceed with the original proof and derive a contradiction.  Thus, we may restrict the items \eqref{bctii}-\eqref{bctiv} from Theorem \ref{clockn} to usco $f$ such that $C_{f}$ has measure zero. 

\smallskip

For the third item, the Cantor set can be defined in $\RCA_{0}$ (see \cite{simpson2}*{p.\ 129}), and similar for `fat' Cantor sets, i.e.\ closed nowhere dense sub-sets of $[0,1]$ that can have measure $1-\eps$ for any $\eps>0$.  
Thus, let $(C_{n})_{n\in \N}$ be a sequence of fat Cantor sets in $[0,1]$ such that $C_{n}$ has measure at least $1-\frac{1}{2^{n}}$ for all $n\in \N$.  
Now let $(O_{n})_{n\in \N}$ be any sequence of dense open sets in $[0,1]$ and note that  $O_{n}\cap V_{n}$ is also open and dense, where $V_{n}=:[0,1]\setminus C_{n}$.  Since $\cap_{n\in \N} V_{n}$ has measure zero, the third item in the theorem applies to $(O_{n}\cap V_{n})_{n\in \N}$, i.e.\ there is $y\in \cap_{n\in\N}\big(O_{n}\cap V_{n}\big )\subset \cap_{n\in \N}O_{n}$, as required for $\BCT_{[0,1]}$.  
\end{proof}
Ironically, the third item of Theorem \ref{deng} suggests that the essence of $\BCT_{[0,1]}$, a fundamental statement about category, in fact involves measure theory.

\smallskip

Fourth, by Theorem \ref{clockn}, the constructive enrichment provided by \emph{oscillation functions} does not change the strength of the semi-continuity lemma.  
Let us consider another constructive enrichment: a \emph{modulus} for usco $f:[0,1]\di \R$ is \emph{any} function $\Psi:[0,1]\di \R^{+}$ such that we have:
\[\textstyle
(\forall k\in \N) (\forall y\in B(x, \Psi(x,k)))( f(y)< f(x)+\frac{1}{2^{k}}   ).
\]  
Given a modulus, the semi-continuity lemma is provable in a weaker system.  
\begin{thm}[$\ACAo$]\label{toblo2}
Let $f:[0,1]\di \R$ be usco with a modulus.  Then $C_{f}\ne \emptyset$.
\end{thm}
\begin{proof}
By definition, for usco $f:[0,1]\di \R$, the set $\{x\in [0,1]: f(x)\geq a\}$ is closed for any $a\in \R$.  
Let $O_{a}$ be the complement of the previous set.    
When given a modulus $\Psi$, one readily defines $Y:[0,1]\di \R$ such that $B(y, Y(y))\subset O_{a}$ in case $y\in O_{a}$.
Hence, $O_{a}$ is R2-open and the Baire category theorem for such sets is provable in $\ACAo$ by \cite{dagsamVII}*{Theorem 7.10}.  
One can now repeat the (modified) proof of item \eqref{bctii} in Theorem \ref{clockn} in $\ACAo$, noting that $\textsf{int}(C_{q})$ has an R2-open representation in light of the equivalence between \eqref{dumbel} and \eqref{dumbel1}.
\end{proof}
Fifth, comparing Theorems \ref{clockn} and \ref{toblo2}, the constructive enrichment provided by oscillation functions does not change the strength of the semi-continuity lemma, while a modulus for usco does make a significant difference.
We now show that the constructive enrichment provided by the usual definition of `Baire 1' is significant.  
\begin{thm}[$\ACAo$]\label{toblo}
Let $f:[0,1]\di \R$ be usco and let $(f_{n})_{n\in \N}$ be a sequence of continuous functions that converges pointwise to $f$ on $[0,1]$.  Then $C_{f}\ne \emptyset$.
\end{thm}
\begin{proof}
We use the proof of \eqref{bcti} $\di $ \eqref{bctii} from Theorem \ref{clockn}.  
In this proof, one defines a sequence of closed and nowhere dense sets $(F_{q})_{q\in \Q}$ such that the union contains $D_{f}$ for lsco $f:[0,1]\di \R$.  Applying $\BCT_{[0,1]}$, the set $C_{f}$ is observed to be non-empty.  
We show that the extra information provided by the Baire 1 representation of $f$ allows us to define an R2-representation for $F_{q}$.  Instead of $\BCT_{[0,1]}$, we can then apply the Baire category theorem for R2-representations, which is provable in $\ACAo$ by \cite{dagsamVII}*{Theorem 7.10}.  In particular, we note that by (the proof of) \cite{dagsamXIV}*{Theorem 2.9}, there is $\Phi:(\R\times \N)\di \R$ such that $\Phi(x, N)$ equals $\inf_{y\in B(x, \frac{1}{2^{N}})}f(y)$ for Baire 1 $f$.   

\smallskip

Let $f:[0,1]\di \R$ be lsco and let $(f_{n})_{n\in \N}$ be a sequence of continuous functions with pointwise limit $f$ on the unit interval.  
Recall the set $F_{q}\subset [0,1]$ defined as those $x\in [0,1]$ satisfying \eqref{ling2}, where the latter is equivalent to \eqref{ling}.
Define $O_{q}:=[0,1]\setminus F_{q}$ and note that `$x\in O_{q}$' is exactly
\be\textstyle \label{cornholio}
f(x)> q \vee (\exists N\in \N){(\forall r\in B(x, \frac{1}{2^{N}})\cap \Q)}( f(r)\leq q  ).
\ee
In case $x\in O_{q}$ satisfies the second disjunct in \eqref{cornholio}, we use $\mu^{2}$ to find the least $N\in \N$ such that $\big(B(x, \frac{1}{2^{N}})\cap \Q\big)\subset O_{q}$.
Now, since $f$ is lsco, the latter implies $B(x, \frac{1}{2^{N}})\subset O_{q}$, as required for a R2-representation.  
In case $x\in O_{q}$ satisfies the first disjunct in \eqref{cornholio}, observe that
\be\label{shalala}\textstyle
f(x)>q \asa (\exists N\in \N)(q<\inf_{y\in B(x, \frac{1}{2^{N}})}f(y)    ), 
\ee
as $f$ is lsco.   Since the infimum in the right-hand side of \eqref{shalala} is given by a function $\Phi:(\R\times \N)\di \R$, we may use $\mu^{2}$ to find the least $N_{0}\in \N$ such that $q<\inf_{y\in B(x, \frac{1}{2^{N_{0}}})}f(y)$.    By the definition of $O_{q}$, we have $B(x, \frac{1}{2^{N_{0}}})\subset O_{q}$, as required for a R2-representation.  Hence, we have obtained a R2-representation of $O_{q}$ based on the case distinction (decidable by $\exists^{2}$) provided by \eqref{cornholio}. 
\end{proof} 
Next, it is well-known that usco functions are Baire 1, but this fact is not provable in $\Z_{2}^{\omega}$ by Corollary \ref{lopi}.
In particular, the previous proof does not go through for usco functions (without a Baire 1 representation) as by \cite{dagsamXIV}*{Theorem 2.32}, the supremum principle for usco functions implies $\NIN_{[0,1]}$.  
\begin{cor}[$\ACAo$]\label{lopi}
The statement 
\begin{center}
\emph{any usco function $f:[0,1]\di \R$ is Baire 1} 
\end{center}
both implies $\BCT_{[0,1]}$ and the principle 
\begin{center}
\textsf{\textup{open:}} \emph{open sets in the unit interval have RM-codes}.
\end{center}
\end{cor}
\begin{proof}
The first part is immediate from Theorem \ref{clockn} and the theorem.  For the second part, a set $C\subset [0,1]$ is closed if and only if $\mathbb{1}_{C}$ is usco.  
Now, in case $x\in O:=[0,1]\setminus C$, we have $(\exists N\in \N)(0= \sup_{y\in B(x, \frac{1}{2^{N}})} \mathbb{1}_{C}(y) )$.  As noted in the proof of the theorem, 
the latter supremum is not just notation but given by a functional in case the function is also Baire 1. Hence, we can use $\mu^{2}$ to define $Y:\R\di \N$ such that $B(x, \frac{1}{2^{Y(x)}})\subset O$ in case $x\in O$.
As $\mu^{2}$ returns the least witness, $\cup_{q\in O\cap \Q}B(q, \frac{1}{2^{Y(q)+1}})$ is an RM-code for $O$, and we are done. 
\end{proof}
The principle \textsf{open} in Corollary \ref{lopi} is `explosive' in that $\FIVE^{\omega}+\open$ proves $\SIX$ (see \cite{dagsamXI}), while $\FIVE^{\omega}$ is $\Pi_{3}^{1}$-conservative over $\FIVE$ (see \cite{yamayamaharehare}).  As an exercise, the first centred principle of Corollary \ref{lopi} is similarly explosive when restricted to usco functions that are continuous almost everywhere.  
We also have the following stronger result in a stronger system.  
\begin{cor}[$\ACAo+\FIVE$]
Let $f:[0,1]\di  [0,1]$ be usco and effectively Baire $n$ for $n\geq 2$.  Then $C_{f}\ne \emptyset$.
\end{cor}
\begin{proof}
The supremum principle for bounded effectively Baire $n$ ($n\geq 2$) functions is provable in $\ACAo+\FIVE$ by \cite{dagsamXIV}*{Theorem 2.22}.  
\end{proof}
Finally, we show that e.g.\ items \eqref{bcti} and \eqref{bct6} of Theorem \ref{clockn} have a proof in $\ACAo$ for Baire 1 functions.  
\begin{thm}[$\ACAo$]\label{flang}
Let $f:[0,1]\di \R$ be Baire 1.  
\begin{itemize}
\item There exists $x\in [0,1]$ such that $f$ is continuous at $x$.
\item There exist $N\in \N$ and $c, d\in [0,1]$ such that $(\forall y\in (c, d))(|f(y)|\leq N+1)$.
\item The oscillation function $\osc_{f}:[0,1]\di \R$ exists.
\end{itemize}
\end{thm}
\begin{proof}
Let $(f_{n})_{n\in \N}$ be a sequence of continuous functions with pointwise limit $f$.  
By \cite{dagsamXIV}*{Cor.\ 2.5}, there is an associated sequence of RM-codes $(\Phi_{n})_{n\in \N}$.  
For the second item, define the closed set $C_{n}$ as follows:
\[
C_{n}:= \{  x\in [0,1]: (\forall m\in \N)(|f_{m}(x)|\leq n ) \}.
\]
Thanks to the RM-codes $\Phi_{n}$ for $f_{n}$, the formula $(\forall m\in \N)(f_{m}(x)\leq n )$ is (second-order) $\Pi_{1}^{0}$.  
Hence, the sets $C_{n}$ have corresponding RM-codes by \cite{simpson2}*{II.5.7}.  Since $[0,1]=\cup_{n\in \N}C_{n}$, the Baire category theorem for RM-codes of open sets (provable in $\RCAo$ by \cite{simpson2}*{II.5.8}), yields that at least one $C_{n}$ is not nowhere dense. 

\smallskip

For the first item, define $D_{m,n}:=\{ x\in [0,1]: (\forall k\geq m)(|f_{k}(x)-f_{m}(x)|\leq \frac{1}{2^{n}}  ) $.
As in the previous paragraph, each $D_{m,n}$ is RM-closed.  To define the interior of the latter sets, put $x\in \textsf{int}(D_{m,n} )$ if and only if $(\exists N\in \N )(\forall r \in B(x, \frac{1}{2^{N}})\cap \Q)(r \in D_{m,n})$.  
Due to the continuity of the $f_{n}$, $x\in \textsf{int}(D_{m,n} )$ is equivalent to 
\[\textstyle
(\exists N\in \N )(\forall y \in B(x, \frac{1}{2^{N}}))(y \in D_{m,n}), 
\]
i.e.\ the `real' definition of the interior of $D_{m,n}$
As in the proof of Theorem~\ref{clockn}, define $F_{m,n}:=D_{m,n}\setminus \textsf{int}(D_{m,n} )$, which is RM-closed and nowhere dense, and note that each point where $f$ is discontinuous, is included in $\cup_{m,n\in \N}F_{m,n}$.  The Baire category theorem for RM-codes (provable in $\RCA_{0}$ by \cite{simpson2}*{II.5.7}) finishes the proof. 

\smallskip

Finally, for the third item. as noted in the proof of Theorem \ref{toblo}, we have access to the supremum operator for Baire 1 functions following \cite{dagsamXIV}*{Theorem 2.9}.
Hence, one readily defines the oscillation function using $\exists^{2}$. 
\end{proof}
In conclusion, we have established some equivalences involving $\BCT_{[0,1]}$ and the semi-continuity lemma.  
We have also shown that the latter become much easier to prove when assuming e.g.\ a Baire 1 representation or a modulus of usco. 

\subsection{Continuity properties of Baire one star functions}\label{bong}
We generalise the equivalences from the previous section.  In particular, we show that Theorem \ref{clockn} goes through if we replace `usco' by `bounded Baire $1^{*}$'.  

\smallskip

First of all, we have the following theorem establishing a decomposition theorem for Baire $1^{*}$ functions.  
\begin{thm}[$\ACAo$]\label{b1sd}
For bounded Baire $1^{*}$ $f:[0,1]\di \R$, there exists usco $g_{0}:[0,1]\di \R$ and lsco $g_{1}:[0,1]\di \R$ such that $f=g_{0}+g_{1}$ on $[0,1]$. 
\end{thm}
\begin{proof}
The proof of \cite{mentoch}*{Lemma 5} goes through with no modification since we assume an upper and lower bound for $f$.
\end{proof}
We study \emph{bounded} Baire $1^{*}$ functions as otherwise the proof of Theorem \ref{b1sd} would need $\NCC$ and $\MCC$ from \cite{dagsamIX}, which are rather non-trivial fragments of the Axioms of Choice (though provable in $\Z_{2}^{\Omega}$).  

\smallskip

Secondly, the previous decomposition theorem gives rise to the following equivalences, where we only treat some basic cases compared to Theorem \ref{clockn}. 
\begin{thm}[$\ACAo$]\label{clockm} The following are equivalent. 
\begin{enumerate}
 \renewcommand{\theenumi}{\alph{enumi}}
\item The Baire category theorem as in $\BCT_{[0,1]}$.\label{b11}
\item For Baire $1^{*}$ $f:[0,1]\di [0,1]$, there is $x\in [0,1]$ where $f$ is continuous.\label{b12}
\item For Baire $1^{*}$ $f:[0,1]\di [0,1]$ which has an oscillation function $\osc_{f}:[0,1]\di \R$, there is a point $x\in [0,1]$ where $f$ is continuous. \label{b13}
\end{enumerate}
\end{thm}
\begin{proof}
Assume $\BCT_{[0,1]}$ and note that Theorem \ref{clockn} implies that for usco $g_{0}$ and lsco $g_{1}$, the sets $C_{g_{0}}$ and $C_{g_{1}}$ are dense, and hence the intersection is non-empty.  
Theorem \ref{b1sd} now implies items \eqref{b12}-\eqref{b13}.  

\smallskip

Now assume item \eqref{b13} and suppose $(X_{n})_{n\in \N}$ is an increasing sequence of nowhere dense sets such that $[0,1]=\cup_{n\in \N}X_{n}$.  
Now consider $h:[0,1]\di \R$ from \eqref{mopi}, which is Baire $1^{*}$ as $f_{\upharpoonright X_{n}}$ is continuous for any $n\in \N$. 
Item \eqref{b13} provides a point of continuity, which yields $x\not \in \cup_{n\in \N}X_{n}$ as in the proof of Theorem \ref{clockn}, as required.  
\end{proof}
By contrast, we have the following theorem.
\begin{thm}[$\ACAo$]\label{toblocor}
For $f:[0,1]\di [0,1]$ in Baire $1^{*}$ and Baire 1, $C_{f}\ne \emptyset$.
\end{thm}
\begin{proof}
The decomposition $f=g_{0}+g_{1}$ from Theorem \ref{b1sd} is provided by the proof of \cite{mentoch}*{Lemma 5}.  By the latter, a Baire 1 representation of $f$ also provides a Baire 1 representation of $g_{0}$ and $g_{1}$.  
Now apply Theorem \ref{toblo} to conclude that $C_{g_{0}}$ and $C_{g_{1}}$ are non-empty (and dense).  Then $C_{g_{0}}\cap C_{g_{1}}\ne \emptyset$, and we are done. 
\end{proof}
Finally, we mention a class between the continuous and the Baire $1^{*}$ functions, namely the Baire 1$^{**}$ functions introduced in \cite{pawla}.
Trivially, Baire $1^{**}$ functions have a point of continuity, i.e.\ we cannot obtain equivalences as in Theorem \ref{clockm}.

\subsection{The uniform boundedness principle}\label{UBP}
In this section, we connect the Baire category theorem and the uniform boundedness principle from \cite{bro}.

\smallskip

First of all, the uniform boundedness principle can be obtained by applying the Baire category theorem to the pointwise boundedness (see Def.\ \ref{lappel}) condition in the former, as done in the second-order proof (\cite{simpson2}*{II.10.8}).
One readily observes that this application of the Baire category theorem does not require continuity, but in fact immedately generalises to lsco functions, as has been established in e.g.\ \cite{bro, kozy, nguyen1}.  
Moreover, by \cite{bro}*{Theorem~1}, one can \emph{characterise} meagreness in terms of the uniform bounded principle for lsco functions.
In this light, the RM-study of the uniform boundedness principle for lsco functions from \cite{bro} is only natural.  

\smallskip

Secondly, we need the following definitions to formulate Theorem \ref{rink}.
\bdefi\label{lappel} A sequence of functions $(f_{n})_{n\in \N}$ is 
\begin{itemize}
\item \emph{\(pointwise\) bounded} on $X$ if $(\forall x\in X)(\exists N\in \N)(\forall n\in \N)(f_{n}(x)\leq N   )$,
\item \emph{uniformly bounded} on $X$ if $(\exists M\in \N)(\forall x\in X,  m\in \N)(f_{m}(x)\leq M   )$. 
\end{itemize}
\edefi
\begin{thm}[$\ACAo$]\label{rink}
The following are equivalent.
\begin{enumerate}
\renewcommand{\theenumi}{\alph{enumi}}
\item The Baire category theorem as in $\BCT_{[0,1]}$.
\item \emph{(Uniform boundedness principle \cite{bro})} A sequence of lsco functions that is pointwise bounded on $[0,1]$, is uniformly bounded on some $(c, d)\subset [0,1]$.\label{franken}
\end{enumerate}
\end{thm}
\begin{proof}
Assume $\BCT_{[0,1]}$ and let $(f_{n})_{n\in \N}$ be a sequence of lsco functions that is pointwise bounded on $[0,1]$.  
By definition, the set $E_{n,N}:=\{x\in [0,1]: f_{n}(x)\leq N\}$ is closed, and so is $F_{N}:=\cap_{n\in \N}E_{n, N}$. 
By pointwise boundedness, we have $[0,1]=\cup_{N\in \N}F_{N}$ and $\BCT_{[0,1]}$ implies there is $N_{0}\in \N$ such that $(c, d)\subset F_{N_{0}}$ for some $c, d\in [0,1]$.
This interval is as required for uniform boundedness in item \eqref{franken}. 

\smallskip

For the reversal, assume item \eqref{franken} and define $\hat{h}:[0,1]\di \R$ as follows:
\be\label{tock}
\hat{h}(x):=
\begin{cases}
0 & x\not\in \cup_{n\in \N}X_{n}\\
n & x\in X_{n} \textup{ and $n\in \N$ is the least such number}
\end{cases},
\ee
where $(X_{n})_{n\in \N}$ is an increasing sequence of closed sets such that $[0,1]=\cup_{n\in \N}X_{n}$. 
Now define $\hat{h}_{n}:[0,1]\di \R$ as follows: $\hat{h}_{n}(x):=\hat{h}(x)$ in case $\hat{h}(x)<n$, and $\hat{h}_{n}(x):=n$ otherwise.
One readily verifies that the functions $\hat{h}, \hat{h}_{n}$ are lsco using the definition.   
Now, $(\hat{h}_{n})_{n\in \N}$ is pointwise bounded on $[0,1]$, as $(\forall x\in X)(\forall n\in \N)(\hat{h}_{n}(x)\leq \hat{h}(x)   )$ by definition.
Apply item \eqref{franken} to find $(c, d)\subset [0,1]$ where $(\hat{h}_{n})_{n\in \N}$ is uniformly bounded, i.e.\ $(\exists M\in \N)(\forall m\in \N, x\in (c, d))(\hat{h}_{m}(x)\leq M )$.  
Taking $m=M+1$, we note that $(c, d)\subset X_{M}$, as required for $\BCT_{[0,1]}$.
\end{proof}
Next, basic properties of metric spaces imply $\NIN$ in light of \cite{dagsamX}*{\S3.2.3}. 
These results are based on the observation that an injection from $[0,1]\di \N$ gives rise to a `strange' metric on $[0,1]$.
We have not been able to obtain similarly `strange' norms on basic vector spaces.   Without such a norm, it seems we cannot obtain an equivalence between $\BCT_{[0,1]}$ and associated theorems 
about Banach spaces.  

\smallskip

Finally, the following theorem is often proved using the Baire category theorem:
\begin{center}
\emph{for continuous $f:[0, +\infty)\di \R$ such that $\lim_{n\di \infty}f(nx)=0$ for all $x\in [0, +\infty)$, we have $\lim_{x\di +\infty}f(x)=0$.}
\end{center}
Like the uniform boundedness principle, the centred statement generalises to lsco functions, which is an interesting exercise for the reader.

\subsection{Continuity properties of cliquish functions}\label{cliq}
We establish equivalences between $\BCT_{[0,1]}$ and basic theorems about cliquish and simply continuous functions (Theorems \ref{clockook} and \ref{flip}), which makes these basic theorems unprovable in $\Z_{2}^{\omega}$. 
By contrast, the associated theorems for quasi-continuous functions are provable in $\ACAo$ (Theorem~\ref{plonkook}).  
We discuss the intimate connection between cliquishness, simple continuity, and quasi-continuity in Remark \ref{donola}.  

\smallskip

First of all, we have the following results involving continuity where we note that item \eqref{la5} deals with \emph{any} function that is totally discontinuous.
We also note that item \eqref{la2} expresses the non-trivial part of the equivalence between `cliquish' and `pointwise discontinuous' on the reals, which was already observed by Dini (\cite{dinipi}*{\S 63}) and is discussed in detail in Remark \ref{donola}.
\begin{thm}[$\ACAo$]\label{clockook} The following are equivalent. 
\begin{enumerate}
 \renewcommand{\theenumi}{\alph{enumi}}
\item The Baire category theorem as in $\BCT_{[0,1]}$.\label{la1}
\item For cliquish $f:[0,1]\di \R$ which has an oscillation function $\osc_{f}:[0,1]\di \R$, there is a point $x\in [0,1]$ where $f$ is continuous. \label{la2}
\item For cliquish $f:[0,1]\di \R$ which has an oscillation function $\osc_{f}:[0,1]\di \R$, there is a point $x\in [0,1]$ where $f$ is lqco. \label{la3}
\item For cliquish and uqco $f:[0,1]\di \R$ which has an oscillation function $\osc_{f}:[0,1]\di \R$, there is a point $x\in [0,1]$ where $f$ is continuous. \label{la4}
\item For totally discontinuous $f:[0,1]\di \R$ with oscillation function $\osc_{f}:[0, 1]\di \R$, there is $N\in \N$ and $c, d\in  [0,1]$ with $\osc_{f}(x)\geq \frac{1}{2^{N}}$ for $x\in (c, d)$.\label{la5}
\end{enumerate}
\end{thm}
\begin{proof}
First of all, assume item \eqref{la2} and fix a decreasing sequence of dense open sets $(O_{n})_{n\in \N}$.  
Define the nowhere dense closed set $X_{n}:= [0,1]\setminus O_{n}$ and consider $h:[0,1]\di \R$ as in \eqref{mopi}.  
Following Theorem \ref{flap}, item \eqref{la2} yields a point of continuity $y\in [0,1]$ of $h$. 
If $\frac{1}{2^{m_{0}+1}}=h(y)>0$, then by definition $y\in X_{m_{0}}$ where $m_{0}\in \N$ is the least such number.  
Since $h$ is continuous at $y$, there is $N\in \N$ such that $z\in B(y, \frac{1}{2^{N}})$ implies $|h(z)-h(y)|<\frac{1}{2^{m_{0}+2}}$.
However, $O_{m_{0}+4}$ is dense in $[0,1]$, implying there is $z_{0}\in B(y, \frac{1}{2^{N}})\cap O_{m_{0}+3}$.  
Since $h(z_{0})\leq\frac{1}{2^{m_{0}+5}}$, we obtain a contradiction.  Hence, we must have $h(y)=0$ and the latter implies $y\in \cap_{n\in \N}O_{n}$ by definition, as required for $\BCT_{[0,1]}$.
As $h$ from \eqref{mopi} is usco by Theorem \ref{flap}, item \eqref{la4} also applies.   For item~\eqref{la3}, one uses the proof of Theorem \ref{deng}.

\smallskip

Secondly, assume $\BCT_{[0,1]}$ and define $D_{m}:= \{ x\in [0,1]:\osc_{f}(x)\geq \frac{1}{2^{m}} \}$, which is closed by definition (or by Theorem \ref{flappie}). 
Suppose $D_{f}=\cup_{n\in \N}D_{n}$ equals $[0,1]$.   By $\BCT_{[0,1]}$, there is $n_{0}\in \N$ such that $D_{n_{0}}$ is not nowhere dense, i.e.\ there is $[a, b]\subset D_{n_{0}} $.
Now fix $x_{0}\in [0,1]$, $\eps_{0}:=\frac{1}{2^{n_{0}+1}}$, and $N_{0}\in \N$ such that $B(x_{0},\frac{1}{2^{N_{0}}})\subset (a, b)$.
Since $f$ is cliquish, there is an interval $(c, d)\subset B(x_{0},\frac{1}{2^{N_{0}}}) $ such that for all $z, w\in (c, d)$, we have $|f(z)-f(w)|<\eps_{0}$.  
The latter implies that $\osc_{f}(z)\leq \eps_{0}$ for $z \in (c, d)$, which however contradicts the fact that $z\in (c, d)\subset D_{n_{0}}$.  Hence, $D_{f}$ does not equal $[0,1]$, as required for item \eqref{la2}.

\smallskip

Fourth, item \eqref{la4} follows since $h$ from \eqref{mopi} is usco by Theorem \ref{flap}. 

\smallskip

Fifth, let $f:[0,1]\di \R$ be as in item \eqref{la5} and define the closed set $D_{k}:=\{ x\in [0,1]: \osc_{f}(x)\geq\frac{1}{2^{k}}\}$. 
By assumption, $[0,1]=\cup_{k\in \N}D_{k}$ and $\BCT_{[0,1]}$ implies there is $k_{0}\in\N$ and $(c, d)\subset [0,1]$ such that $(c, d)\subset D_{k_{0}}$, as required.
For the reversal, assume item \eqref{la5} and let $f:[0,1]\di \R$ be cliquish with oscillation function $\osc_{f}:[0,1]\di \R$.  
In case $C_{f}\ne \emptyset$, we are done in light of item \eqref{la2}. 
In case $C_{f}=\emptyset$, $f$ is totally discontinuous and we can apply item \eqref{la5}.  
Let $N_{0}\in \N$ and $c, d\in  [0,1]$ be such that $\osc_{f}(x)\geq \frac{1}{2^{N_{0}}}$ for $x\in (c, d)$.  
Now fix $\eps_{0}>0$ and $x_{0}\in (c, d)$ such that  $ \eps_{0}<\frac{1}{2^{N_{0}}}$ and $B(x_{0}, \eps_{0})\subset (c,d)$; as $f$ is cliquish, there is an interval $(a, b)\subset B(x_{0}, \eps_{0})$ such that for all $x, y\in (a, b)$, we have $|f(x)-f(y)|<\eps_{0}$.  However, this implies that $\osc_{f}(z)\leq \frac{1}{2^{N_{0}}}$ for $z\in (a, b)$, contradicting $\osc_{f}(z)>\frac{1}{2^{N_{0}}}$ by $z\in (a, b)\subset (c, d)$.
 \end{proof}
Secondly, the next theorem is motivated by the \emph{supremum principle} for cliquish functions; this principle implies $\NIN_{[0,1]}$ by \cite{dagsamXIV}*{Theorem 2.32} and states the existence of $F:\Q^{2}\di \R$ such that $F(p, q)=\sup_{x\in [p,q]}f(x)$ for $p, q\in [0,1]\cap \Q$.
Perhaps surprisingly, the above equivalences for $\BCT_{[0,1}$ still go through if we restrict to cliquish functions that have a supremum functional.
\begin{thm}[$\ACAo+\QFAC^{0,1}$]\label{diender}
The following are equivalent.  
\begin{enumerate}
 \renewcommand{\theenumi}{\alph{enumi}}
\item The Baire category theorem as in $\BCT_{[0,1]}$.
\item For cliquish $f:[0,1]\di \R$ with a supremum functional and oscillation functional, there is $x\in [0,1]$ where $f$ is continuous.\label{super1}
\end{enumerate}
\end{thm}
\begin{proof}
The first item implies the second item by Theorem \ref{clockook}.  
Now assume the second item and let $(X_{n})_{n\in \N}$ be an increasing sequence of closed nowhere dense sets such that $[0,1]=\cup_{n\in \N}X_{n}$. 
Define $Z(x)$ as the least $n\in \N$ such that $x\in X_{n}$ and $e(x):= \sum_{n=0}^{Z(x)+1}\frac{x^{n}}{n!} $, implying $e(x)<e^{x}$ for $x\in [0,1]$ and $\sup_{x\in [p,q]}e(x)=e^{q}$ by definiton.
To show that $\lambda x.e(x)$ is cliquish, fix $x_{0}\in [0,1]$ and consider 
\[
|e(y)-e(z)|\leq |e(y)-e^{y}|+|e^{y}-e^{z}|+|e(z)-e^{z}|, 
\]
to which the usual `epsilon over three' trick applies as follows: the middle term can be made arbitrarily small thanks to the uniform continuity of $\lambda x.e^{x}$ on $[0,1]$.  
The other two terms can be made arbitrarily small by picking large enough $k\in \N$ and picking $x_{1}\in O_{k}:=[0,1]\setminus X_{k}$ close enough to $x_{0}$ (using the density of the sets $O_{n}$). 
Indeed, $B(x_{1}, \frac{1}{2^{N_{0}}})\subset O_{k}$ for some $N_{0}\in \N$, implying that $Y(w)\geq k$ for $w$ in the former ball, i.e.\ $e(w)$ approaches $e^{w}$ as $k$ grows.  
Since $\lambda x.e(x)$ is totally discontinuous, item \eqref{super1} yields a contradiction. Hence, there must be $y\not \in \cup_{n\in \N}X_{n}$ as required for $\BCT_{[0,1]}$, and we are done.  
\end{proof}
Thirdly, as discussed in Remark \ref{donola}, the sico functions yields a class between the quasi-continuous and cliquish ones.  
As it happens, we can restrict to sico functions in Theorem \ref{clockook} if we deal with \emph{sequentially}\footnote{A function $f:[0,1]\di \R$ is \emph{sequentially sico} if for any sequence of open sets $(G_{n})_{n\in \N}$, there are sequences $(O_{n})_{n\in \N}$ and $(N_{n})_{n\in \N}$ of open and nowhere dense sets with $f^{-1}(G_{n})=O_{n}\cup{N_{n}}$.  The proof that sico functions are cliquish (see \cite{nieuwbronna2}*{Theorem 2.1 and 2.2}) goes through without extra induction or fragments of the Axiom of Choice if we assume \emph{sequentially} sico functions.} sico functions.  
This restriction is interesting as e.g.\ Thomae's function from \eqref{thomae} is cliquish but not sico.  
\begin{thm}[$\ACAo+\QFAC^{0,1}$]\label{flip}
The following are equivalent.  
\begin{enumerate}
\renewcommand{\theenumi}{\alph{enumi}}
\item The Baire category theorem as in $\BCT_{[0,1]}$.
\item For sequentially sico $f:[0,1]\di \R$, there is $x\in [0,1]$ where $f$ is continuous.\label{kuper1}
\end{enumerate}
\end{thm}
\begin{proof}
First of all, sico functions are cliquish (say in $\ZFC$) while the proof of \cite{nieuwbronna2}*{Theorem 2.1} immediately establishes that \emph{sequentially} sico functions are cliquish, working over $\ACAo+\BCT_{[0,1]}$.
Hence, item \eqref{kuper1} immediately follows from $\BCT_{[0,1]}$ via Theorem~\ref{clockook}.

\smallskip

Secondly, let $(X_{n})_{n\in \N}$ be an increasing sequence of closed and nowhere dense sets such that $[0,1]=\cup_{n\in \N}X_{n}$ and consider $h:[0,1]\di \R$ as in \eqref{mopi}.  
Fix $G\subset \R$ and consider the following two cases which show that $h$ is (sequentially) sico.
\begin{itemize}
\item In case $(\forall n\in \N)(\exists m> n)( \frac{1}{2^{m+1}}\in G)$, we have $h^{-1}(G)=\cup_{n\in \N}X_{n}=[0,1]$. 
\item In case $n_{0}\in \N$ is the least $n$ with $(\forall m> n)( \frac{1}{2^{m+1}}\not\in G)$, then $h^{-1}(G)= X_{n_{0}}$.  
\end{itemize}
Note that $\mu^{2}$ can find the number $n_{0}$ from the second item and decide which case holds, i.e.\ $h$ is also sequentially sico. 
Following Theorem \ref{fronkn}, item \eqref{kuper1} now provides a point of continuity for $h$, and $\BCT_{[0,1]}$ readily follows. 
\end{proof}
Folllowing \cite{nieuwbronna2}*{\S3}, it seems that Theorem \ref{flip} can be generalised to \emph{pseudo-continuity}, a notion equivalent to cliquishness where the inverse image of open sets is the union of an open and a meagre set.  

\smallskip

Fourth, while quasi-continuous functions can be quite `wild' in light of Remark~\ref{donola}, their basic properties are readily obtained in relatively weak systems by Theorem \ref{plonkook}.  
Comparing to Theorem \ref{clockook}, the latter result is all the more surprising since `cliquish' is quite close to `quasi-continuous', the former essentially being the closure of the latter under sums, as also discussed in Remark \ref{donola}.
\begin{thm}[$\ACAo$]\label{plonkook}
For quasi-continuous $f:[0,1]\di \R$, we have the following: 
\begin{itemize}
\item the set $C_{f}$ exists and is dense,
\item there is a sequence $(D_{n})_{n\in \N}$ of R2-closed nowhere dense sets with $D_{f}=\cup_{n\in \N}D_{n}$,
\item the oscillation function $\osc_{f}:[0,1]\di \R$ exists. 
\end{itemize}
\end{thm}
\begin{proof}
Fix quasi-continuous $f:[0,1]\di \R$ and use $\exists^{2}$ to define $x\in O_{m}$ in case 
\be\label{obvio}\textstyle
(\exists N\in \N)(\forall q, r\in  B(x, \frac{1}{2^{N}})\cap \Q)( |f(q)-f(r)|\leq \frac{1}{2^{m}}  ).  
\ee
By quasi-continuity, the formula \eqref{obvio} is equivalent to 
\be\label{poiuy}\textstyle
(\exists M\in \N)(\forall w, z\in  B(x, \frac{1}{2^{M}}))( |f(w)-f(z)|\leq \frac{1}{2^{m}}  ),   
\ee
which immediately implies that $O_{m}$ is open.  Apply $\mu^{2}$ to \eqref{obvio} to obtain an R2-representation of $O_{m}$.  
Now define $D_{m}:= [0,1]\setminus O_{m}$ and note that $D_{f}:=\cup_{m\in \N}D_{m}$ contains all points where $f$ is discontinuous (essentially by definition). 
By \cite{dagsamVII}*{Theorem 7.10}, $\ACAo$ proves $\BCT_{[0,1]}$ for R2-open sets, i.e.\ the first and second item now follow.  
Finally, by \cite{dagsamXIV}*{Theorem 2.9}, the supremum and infimum of $f$ exist (uniformly on any interval) given $\exists^{2}$.  Hence, the oscillation function as in the third item is readily defined, where we note that $D_{m}$ is by definition $\{x\in [0,1]:\osc_{f}(x)\geq \frac{1}{2^{m}}\}$, i.e.\ the usual definition.
\end{proof}
Finally, we discuss general properties of quasi-continuous and cliquish functions.  
\begin{rem}\label{donola}\rm
First of all, the class of cliquish functions on $\R$ has nice technical and conceptual properties, as witnessed by the following list.
\begin{itemize}
\item The class of cliquish functions is exactly the class of sums of quasi-continuous functions (\cite{bors, quasibor2, malin}). In particular, cliquish functions are closed under sums while quasi-continuous ones are not.
\item The pointwise limit (if it exists) of quasi-continuous functions, is always cliquish (\cite{holausco}*{Cor.\ 2.5.2}).
\item Similar to Baire 1 functions, cliquish functions are \emph{exactly} the limits of lqco and uqco sequences (see \cite{ewert2}). 
\item The set $C_{f}$ is dense in $\R$ if and only if $f:\R\di \R$ is cliquish (\cites{bors, dobo, dinipi}), i.e.\ the cliquish functions are exactly the pointwise discontinuous ones. 
The forward implication is trivial, while the other one is not by Theorem \ref{clockook}.
\item The simply continuous functions on $\R$ yield a class intermediate between the quasi-continuous and cliquish ones (see \cite{bors}). 
\end{itemize}
Secondly, quasi-continuous functions can be quite `wild': if $\mathfrak{c}$ is the cardinality of $\R$, there are $2^{\mathfrak{c}}$ non-measurable quasi-continuous $[0,1]\di \R$-functions and $2^{\mathfrak{c}}$ measurable non-Borel quasi-continuous $[0,1]\di [0,1]$-functions (see \cite{holaseg}).  Also, the class of quasi-continuous functions is closed under taking \emph{transfinite} limits (\cite{nieuwebronna}).  
\end{rem}

\subsection{Continuity properties of Baire one functions}\label{1227}
We establish equivalences between $\BCT_{[0,1]}$ and basic theorems about Baire 1 functions (Theorem \ref{clockal}), which makes the latter unprovable in $\Z_{2}^{\omega}$. 
By Theorem \ref{toblo}, basic theorems about Baire~1 functions are provable in weak systems, explaining why we need equivalent definitions of Baire 1 as in Definition \ref{frag}.

\smallskip

First of all, there are a surprisingly large number of rather diverse equivalent definitions of `Baire 1' on the reals (\cites{leebaire,beren,koumer}), including the following ones by \cite{koumer}*{Theorem~2.3} and \cite{kura}*{\S34, VII}.   
\bdefi\label{frag} For a function $f:[0,1]\di \R$, we say that 
\begin{itemize}
\item $f$ is \emph{fragmented} if for any $\eps>0$ and closed $C\subset [0,1]$, there is non-empty relatively\footnote{For $A\subseteq B\subset \R$, we say that \emph{$A$ is relatively open \(in $B$\)} if for any $a\in A$, there is $N\in \N$ such that $B(x, \frac{1}{2^{N}})\cap B\subset A$.  Note that $B$ is always relatively open in itself.} open $O\subset C$ such that $\textup{\textsf{diam}}(f(O))<\eps$,
\item $f$ is \emph{$B$-measurable of class 1} if for every open $Y\subset \R$, the set $f^{-1}(Y)$ is $\F_{\sigma}$, i.e.\ a union over $\N$ of closed sets. 
\end{itemize}
\edefi
The \emph{diameter} of a set $X$ of reals is defined as usual, namely $\textup{\textsf{diam}}(X):=\sup_{x, y \in X}|x-y|$, where the latter supremum need not exist for Definition \ref{frag}.  

\smallskip

Secondly, fragments of the induction axiom are sometimes used in an essential way in second-order RM (see e.g.\ \cite{neeman}).  
An important role of induction is to provide `finite comprehension' (see \cite{simpson2}*{X.4.4}).
We shall need the following fragment of finite comprehension for Theorem \ref{clockal} and in Section \ref{related}.  
\begin{princ}[$\IND_{\R}$]
For $F:(\R\times \N)\di \N, k\in \N$, there is $X\subset \N$ such that
\[
(\forall n\leq k)\big[ (\exists x\in \R)(F(x, n)=0)\asa n\in X\big].
\]
\end{princ}
Thirdly, we have the following results for Baire 1 functions as in Def.\ \ref{frag}. 
\begin{thm}[$\ACAo+\IND_{\R}$]\label{clockal} The following are equivalent. 
\begin{enumerate}
 \renewcommand{\theenumi}{\alph{enumi}}
\item The Baire category theorem as in $\BCT_{[0,1]}$.\label{fra1}
\item For fragmented $f:[0,1]\di \R$ which has an oscillation function $\osc_{f}:[0,1]\di \R$, there is a point $x\in [0,1]$ where $f$ is continuous. \label{fra2}
\item For $B$-measurable of class 1 and cliquish $f:[0,1]\di \R$ which has an oscillation function $\osc_{f}:[0,1]\di \R$, there is $x\in [0,1]$ where $f$ is continuous.\label{fra3}
\end{enumerate}
\end{thm}
\begin{proof}
First of all, we show that fragmented functions are cliquish (in $\ACAo$); Theorem \ref{clockook} then establishes the implication from $\BCT_{[0,1]}$ to item \eqref{fra2}.
Let $f:[0,1]\di \R$ be fragmented and fix $\eps>0$, $N\in \N$, and $x\in [0,1]$.  
For closed $C:=[x_{0}-\frac{1}{2^{N_{0}}}, x_{0}+\frac{1}{2^{N_{0}}}]$, there is a relatively open $O\subset C$ with $\textup{\textsf{diam}}(f(O))<\eps$.  
Since $C$ is a closed interval, the relative openness of $O$ implies that there are $a, b\in \R$ with $(a, b)\subset O$.  
Since $\textup{\textsf{diam}}(f(O))<\eps$, we have $|f(x)-f(y)|<\eps$ for $x, y\in (a, b)\subset B(x, \frac{1}{2^{N_{0}}})$, i.e.\ $f$ is cliquish. 

\smallskip

Secondly, we show that item \eqref{fra2} implies $\BCT_{[0,1]}$.  
To this end, let $(X_{n})_{n\in \N}$ be an increasing sequence of closed nowhere dense sets in $[0,1]$ such that $[0,1]=\cup_{n\in \N}X_{n}$. 
We now show that $h:[0,1]\di \R$ from \eqref{mopi} is fragmented.  
Fix $\eps >0$ and closed $C\subset [0,1]$ and consider the following case distinction.
\begin{itemize}
\item In case $C\subset X_{0}$, put $O:= C$ and note that $O$ is relatively open (in $C$).  Moreover, $f(x)=\frac{1}{2}$ for $x\in O\subset X_{0}$, i.e.\ Definition \ref{frag} is trivially satisfied.  
\item In case $C\subset \cup_{k\leq k_{0}} X_{k}$ for (minimal) $k_{0}>0$, put $O:= C\cap O_{k_{0}-1}$ where $O_{k}:=[0,1]\setminus X_{k}$.  Since $O_{k_{0}-1}$ is open, $O$ is relatively open (in $C$).  
Similar to the previous item, $f(x)=\frac{1}{2^{k_{0}}}$ for $x\in O$, i.e.\ Definition~\ref{frag} is again trivially satisfied.  
\item For the remaining case, since $C\cap X_{k}\ne \emptyset $ for any $k\in \N$, choose $k_{0}\in \N\setminus \{0\}$ with $\frac{1}{2^{k_{0}}}<\eps$ and define $O:= C\cap O_{k_{0}-1}$.  
Similar to the previous item, $O_{k_{0}-1}$ is open and $O$ is relatively open (in $C$).  
By assumption, $f(x)< \frac{1}{2^{k_{0}}}$ for $x\in O\subset  [0,1]\setminus X_{k_{0}}$.  Hence, $\textup{\textsf{diam}}(f(O))\leq \frac{1}{2^{k_{0}}}<\eps$, as required. 
\end{itemize}
Hence $h$ from \eqref{mopi} is fragmented while the second item seems to need $\IND_{\R}$ to guarantee that $k_{0}$ is minimal. 
We also obtain an oscillation function by Theorem~\ref{fronkn}, i.e.\ we may apply item \eqref{fra2} to show that $h$ is continuous at some $x_{0}\in [0,1]$.  
As in the proof of Theorem~\ref{clockook}, we obtain $\BCT_{[0,1]}$, and we are done.

\smallskip

Thirdly, item \eqref{fra3} implies $\BCT_{[0,1]}$ as follows: let $(X_{n})_{n\in \N}$ be an increasing sequence of closed nowhere dense sets in $[0,1]$ such that $[0,1]=\cup_{n\in \N}X_{n}$. 
Now consider $h$ from \eqref{mopi} and note that for any $G\subset \R$, $h^{-1}(G)$ is the union of all $X_{n}$ such that $\frac{1}{2^{n+1}}\in G$, i.e.\ $\F_{\sigma}$ by definition. 
Together with Theorem \ref{flap}, we observe that item \eqref{fra3} applies to $h$; as in the proof of Theorem~\ref{clockook}, we obtain $\BCT_{[0,1]}$.
The other direction is immediate by Theorem \ref{clockook}. 
\end{proof}
\begin{cor}[$\ACAo$]\label{tokkiecor}
The following statement 
\begin{center}
\emph{any fragmented function $f:[0,1]\di \R$ is Baire 1} 
\end{center}
implies $\BCT_{[0,1]}$ and the statement \textsf{\textup{open}} from Corollary \ref{lopi}.
\end{cor}
\begin{proof}
The first part is immediate from Theorem \ref{flang} and the theorem.  For the second part, one readily shows that for closed $C\subset [0,1]$, the function $\mathbb{1}_{C}$ is fragmented.
As in the proof of Corollary \ref{lopi}, one obtains an RM-code for $C$. 
\end{proof}
Finally, we have not been able to obtain similar results for other equivalent definitions of Baire 1 functions on the reals.

\subsection{On Volterra's early work}\label{vork}
In this section, we connect the Baire category theorem and Volterra's early work as sketched in Section \ref{vintro}, including a version of Blumberg's theorem.  
As a result, Volterra's theorem and corollary cannot be proved in $\Z_{2}^{\omega}$ while the restrictions to quasi-continuous functions are provable in $\ACAo$ by Theorem \ref{flonkqua}. 

\smallskip

First of all, we have the following theorem, where we recall that `cliquish' and `pointwise discontinuous' are equivalent on the reals (see Remark \ref{donola}).
However, in light of Theorem \ref{clockook}, this equivalence cannot be proved in $\Z_{2}^{\omega}$.  

\smallskip

As in the previous sections, we observe a certain robustness for our results: we could replace `usco' by any notion between the latter and cliquishness. 
\begin{thm}[$\ACAo$]\label{flonk}
The following are equivalent.
\begin{enumerate}
\renewcommand{\theenumi}{\alph{enumi}}
\item The Baire category theorem $\BCT_{[0,1]}$.\label{volare0}
\item \emph{Volterra's theorem}: there do not exist two \emph{pointwise discontinuous} $f, g:[0,1]\di \R$ with associated oscillation functions for which the continuity points of one are the discontinuity points of the other, and vice versa.\label{volare1337}
\item \emph{Volterra's theorem} for `pointwise discontinuous' replaced by `cliquish'.\label{volare2}
\item \emph{Volterra's theorem} for `pointwise discontinuous' replaced by `usco'.\label{volare1} 
\item \emph{Volterra's corollary}: there is no function with an oscillation function that is continuous on $\Q\cap[0,1]$ and discontinuous on $[0,1]\setminus\Q$.\label{volare4}
\item \emph{Volterra's corollary} restricted to usco functions.\label{volare3}
\item Generalised Volterra's corollary \(Theorem \ref{dorki}\) for usco functions and enumerable $D$ \(or: countable $D$, or: strongly countable $D$\). \label{volare5}
\item Blumberg's theorem \(\cite{bloemeken}\) restricted to usco \(or cliquish with an oscillation function\) functions on $[0,1]$.\label{volare6}
\end{enumerate}
The equivalences remain correct if we replace `usco' by `Baire $1^{*}$'.  
\end{thm}
\begin{proof}
First of all, to prove item \eqref{volare1} using $\BCT_{[0,1]}$, suppose $f, g:[0,1]\di \R$ are usco and such that $C_{f}=D_{g}$ and $C_{g}=D_{f}$, noting that 
these sets exist by Theorem~\ref{plonk}.  Now consider the (closed and nowhere dense) sets $F_{q}$ from the proof of Theorem~\ref{clockn} and let $H_{q}$ be the associated (closed and nowhere dense) sets for $g$.
Since $D_{f}\subset \cup_{q\in \Q}F_{q}$ and $D_{g}\subset \cup_{q\in \Q}H_{q} $, we have $ [0,1]=C_{f}\cup D_{f}=D_{g}\cup D_{f}=\cup_{q\in \Q}(F_{q}\cup H_{q})$, which contradicts $\BCT_{[0,1]}$.

\smallskip

A similar proof goes through for item \eqref{volare2}, namely by defining $D_{k}^{f}:= \{x\in [0,1]:  \osc_{f}(x)\geq \frac{1}{2^{k}}\}$ and noting that $[0,1]=\cup_{n\in \N}(D_{n}^{f}\cup D_{n}^{g}) $ where the latter sets 
are nowhere dense as in the proof of Theorem \ref{clockook}.  Items \eqref{volare4} and \eqref{volare3} follow from items \eqref{volare2} and \eqref{volare1} as Thomae's function from \eqref{thomae} is continuous on the irrationals and discontinuous on $\Q$; this function is also usco and cliquish.  By definition, pointwise discontinuous functions are cliquish, i.e.\ item \eqref{volare1337} also follows.

\smallskip

Secondly, we now derive $\BCT_{[0,1]}$ from Volterra's corollary as in item \eqref{volare4} or \eqref{volare3}.  
To this end, let $(X_{n})_{n\in \N}$ be a sequence of closed nowhere dense sets such that $[0,1]\setminus \Q=\cup_{n\in \N}X_{n}$.   
Now consider $h:[0,1]\di \R$ as in \eqref{mopi}, which is as required for item \eqref{volare4} or \eqref{volare3} by Theorems \ref{fronkn} and \ref{flap}. 
By Theorem \ref{fronkn}, $h$ is also continuous at any $q\in \Q\cap [0,1]$ (since $h(q)=\osc_{h}(q)=0$) and discontinuous at any $y\in [0,1]\setminus \Q$ (since $h(y)=\osc_{h}(y)>0$).  This contradicts Volterra's corollary (for usco or cliquish functions), and $\BCT_{[0,1]}$ follows.  The same argument works for item \eqref{volare1337} using Thomae's function and $h:[0,1]\di \R$ as the latter is pointwise discontinuous, namely continuous at every rational, by assumption.

\smallskip

Thirdly, we only need to show that $\BCT_{[0,1]}$ implies item \eqref{volare5}, as $\Q$ is enumerable and dense. To this end, proceed as in the previous paragraph, replacing $\Q$ by a countable dense set $D\subset [0,1]$ wherever relevant.  Note in particular that $D=\cup_{n\in \N} \{x\in D: Y(x)= n\}$ suffices to replace an enumeration of $\Q$ if $Y:[0,1]\di \N$ is injective on $D$. 

\smallskip

Fourth, $\BCT_{[0,1]}$ implies items \eqref{volare6} by Theorem \ref{clockn} since $D=C_{f}$ is as required.  
To prove \eqref{volare6}$\di\BCT_{[0,1]}$, let $(X_{n})_{n\in \N}$ be a sequence of closed and nowhere dense sets in $[0,1]$ and consider $h:[0,1]\di \R$ from \eqref{mopi}, which is usco and cliquish by Theorem \ref{flap}.  
Let $D$ be the dense set provided by item~\eqref{volare6} and consider $y\in D$.  In case $h(y)\ne 0$, say $h(y)\geq \frac{1}{2^{k_{0}}}$, note that $O_{k_{0}}\cap D$ is dense in $[0,1]$. 
Hence, for any $N\in \N$ there is $z\in B(x, \frac{1}{2^{N}})\cap D$ with $h(z)\leq\frac{1}{2^{k_{0}+1}}$, i.e.\ $h_{\upharpoonright D}$ is not continuous (on $D$).  
This contradiction implies that $h(y)=0$, meaning $y\in [0,1]\setminus\cup_{n\in \N}X_{n}$, i.e.\ $\BCT_{[0,1]}$ follows.  
The final sentence is immediate in light of Theorem \ref{b1sd}.
\end{proof}
We could use the notion \emph{height countable} (see \cite{dagsamX}) for item \eqref{volare5} of Theorem~\ref{flonk}, but then we need (at least) extra induction to accommodate finite sets.

\smallskip
\noindent
Next, we show that Volterra's theorems from \cite{volaarde2} and related results are readily proved if we restrict to quasi-continuous functions. 
\begin{thm}[$\ACAo$]\label{flonkqua}~
\begin{enumerate}
\renewcommand{\theenumi}{\alph{enumi}}
\item \emph{Volterra's theorem}: there do not exist two quasi-continuous functions defined on the unit interval for which the continuity points of one are the discontinuity points of the other, and vice versa.\label{qolare1}
\item \emph{Volterra's corollary}: there is no quasi-continuous function that is continuous on $\Q\cap[0,1]$ and discontinuous on $[0,1]\setminus\Q$.\label{qolare2}
\item Generalised Volterra's corollary \(Theorem \ref{dorki}\) for quasi-continuous functions and enumerable $D$ \(or: countable $D$, or: strongly countable $D$\). \label{qolare3}
\item Blumberg's theorem \(\cite{bloemeken}\) restricted to quasi-continuous functions on $[0,1]$.\label{qolare4}
\end{enumerate}
\end{thm}
\begin{proof}
For item \eqref{qolare1}, for quasi-continuous $f, g:[0,1]\di \R$, the sets $D_{f}$ and $D_{g}$ exist and be can be expressed as unions $\cup_{n\in \N}D_{n}^{f}$ and $\cup_{m\in \N}D_{m}^{g}$ of R2-closed nowhere dense sets by Theorem \ref{plonkook}.  
Hence, in case $C_{f}=D_{g}$ and $C_{g}=D_{f}$, we have $[0,1]=C_{f}\cup D_{f}=D_{g}\cup D_{f}=\bigcup_{n, m\in \N} D_{n}^{f}\cup D_{m}^{g}  $.
However, the Baire category theorem for R2-closed sets is provable in $\ACAo$ (\cite{dagsamVII}*{Theorem 7.10}), a contradiction. 
The other items are proved in the same way. 
\end{proof}
\noindent
In conclusion, we have established equivalences involving $\BCT_{[0,1]}$ and Volterra's early work (and related results), rendering the latter unprovable in $\Z_{2}^{\omega}$.
The latter deal with pointwise discontinuous function, which are exactly the cliquish ones by Remark \ref{donola}. 
Also by the latter, cliquish functions are closely related to quasi-continuous functions, but Volterra's results for quasi-continuous functions are readily proved in $\ACAo$ by Theorem \ref{flonkqua}.
In contrast to the close relation between cliquishness and quasi-continuity, there is a great divide, some might say `abyss', between $\ACAo$ and $\Z_{2}^{\omega}$. 
We have no explanation for this phenomenon at the moment and welcome any suggestions.

\section{Tao's pigeonhole principle for measure spaces}\label{related}
\subsection{Introduction}
In this section, we establish the equivalences sketched in Section \ref{vower} for $\PHP_{[0,1]}$, i.e.\ Tao's {pigeonhole principle for measure spaces} from \cite{taoeps}. 
We also study $\WBCT_{[0,1]}$, a `hybrid' principle between $\PHP_{0,1}$ and $\BCT_{[0,1]}$ which is interesting in its own right, also from a historical point of view. 

\smallskip

First of all, it is trivial that `measure zero' implies `nowhere dense'.  Replacing the latter by the former in $\BCT_{[0,1]}$, we obtain the weaker $\WBCT_{[0,1]}$ (see Principle \ref{WBCTI} in Section \ref{pix}) and the equivalences from the previous section then go through \emph{mutatis mutandis}. 
Nonetheless, this weakening is more than \emph{spielerei} as $\WBCT_{[0,1]}$ is connected to Hankel's 1870 work on Riemann integration (\cite{hankelwoot}).
In a nutshell, Hankel formulates necessary and sufficient conditions for the Riemann integral akin to the Vitali-Lebesgue theorem from Section \ref{vower}.  
While Hankel's theorem is correct, the proof is not: he incorrectly tries to reverse the implication that a Riemann integrable function is pointwise discontinuous, leading to Smith's counterexample based on Cantor sets (\cite{snutg}) and Ascoli's proof of Hankel's theorem (\cite{ascoli1}).  
Equivalences involving $\WBCT_{[0,1]}$ and Hankel's results for usco and cliquish functions are studied in Section \ref{pix}, including the fundamental theorem of calculus.

\smallskip

Secondly, $\WBCT_{[0,1]}$ only replaces `nowhere dense' by `measure zero' in the antecedent; Tao's {pigeonhole principle for measure spaces} from \cite{taoeps} ensues if we perform this replacement in the consequent as well.  In Section \ref{duif}, we study the RM of the latter principle, called $\PHP_{[0,1]}$, obtaining an equivalence to the Vitali-Lebesgue theorem (including restrictions to usco and cliquish functions).  
As will become clear, the RM of $\PHP_{[0,1]}$ is rather similar to the RM of $\WBCT_{[0,1]}$, which is in turn similar to the RM of $\BCT_{[0,1]}$, despite the fundamental differences between measure and category.  
Finally, the restrictions of the aforementioned theorems to \emph{quasi-continuous} or \emph{Baire 1} functions, are provable in $\ACAo$ by Theorem \ref{qvl}.
Recall that quasi-continuity is very close to cliquishness by Remark \ref{donola}.

\subsection{Measure-theoretic Baire category theorem}\label{pix}
We introduce a `measure-theoretic' version of the Baire category theorem and obtain some equivalences to well-known theorems, including (variations of) the fundamental theorem of calculus and the Vitali-Lebesgue theorem from Section \ref{vower}.

\smallskip

First of all, one part of the Vitali-Lebesgue theorem is readily proved. 
\begin{thm}[$\ACAo+\QFAC^{0,1}$]
A Riemann integrable $f:[0,1]\di \R$ is bounded.
\end{thm}
\begin{proof}
Suppose $f:[0,1]\di \R$ is Riemann integrable and unbounded, i.e.\ $(\forall n\in \N)(\exists x\in [0,1])(|f(x)|>n)$. 
Apply $\QFAC^{0,1}$ to obtain a sequence $(x_{n})_{n\in \N}$ with $|f(x_{n})|>n$ for $n\in \N$.  
By sequential compactness (see \cite{simpson2}*{III.2.2}), there is a convergent sub-sequence $(y_{n})_{n\in \N}$, say with limit $y\in [0,1]$.  
In particular, $f(y_{n})$ is arbitrarily large as $y_{n}$ approaches $y$.  Thus, for any Riemann sum $\sum_{i=0}^{k} f(t_{i}) (x_{i+1}-x_{i})  $, consider the point $t_{j}$ such that $y\in [x_{j}, x_{j+1}]$ and the latter contains infinitely many elements of the sequence $(y_{n})_{n\in \N}$.
Changing this $t_{j}$ to $y_{m}$ for large enough $m\in \N$, the Riemann sum can be arbitrarily large, a contradiction.
\end{proof}
Similarly, $\ACAo+\QFAC^{0,1}$ establishes that a continuous almost everywhere bounded function, is Riemann integrable on the unit interval;
the proof of \cite{brownrie}*{Lemma} is readily formalised, as it is based (only) on Riemann sums, i.e.\ avoiding the Darboux integral and Lebesgue measure.  
Alternative proofs surely exist in the modern literature, noting that \cite{brownrie} is almost 100 years old.  As an aside, the Darboux integral is not directly available as the supremum principle for 
Riemann integrable functions is hard\footnote{As observed below Corollary \ref{lopi}, the principle $\open$ in the latter is `explosive' in that $\FIVE^{\omega}+\open$ proves $\SIX$ (\cite{dagsamXI}), while $\FIVE^{\omega}$ is $\Pi_{3}^{1}$-conservative over $\FIVE$ (\cite{yamayamaharehare}).  The supremum principle for Riemann integrable functions (even with a modulus) behaves the same. 
}  to prove by the results in \cite{dagsamXIV}*{\S2.8}. 

\smallskip

Secondly, let $\WBCT_{[0,1]}$ be the restriction of $\BCT_{[0,1]}$ to measure one sets.  
\begin{princ}[$\WBCT_{[0,1]}$]\label{WBCTI}
If $ (O_n)_{n \in \N}$ is a decreasing sequence of measure one open sets of reals in $[0,1]$, then $ \bigcap_{n \in\N } O_n$ is non-empty.
\end{princ} 
Recall that, like in second-order RM, statements like `$E\subset [0,1]$ has measure $0$' can be made without introducing the Lebesgue measure (see Definition \ref{char}).  
The RM of $\WBCT_{[0,1]}$ seems to require the principle $\IND_{\R}$ introduced in Section \ref{1227}.

\smallskip

Thirdly, we have the following theorem where we note that Hankel studies Riemann integration using oscillation functions in \cite{hankelwoot}. 
We also invite the reader to compare item \eqref{pono2} to the main topic of \cite{hahakaka}, which we discuss further in Section \ref{biggerer}. 
\begin{thm}[$\ACAo+\IND_{\R}+\QFAC^{0,1}$]\label{duck5}
The following are equivalent.  
\begin{enumerate}
\renewcommand{\theenumi}{\alph{enumi}}
\item The weak Baire category theorem $\WBCT_{[0,1]}$.
\item \(Hankel \cite{hankelwoot} \& Dini \cite{dinipi}*{\S188}\) For Riemann integrable $f:[0,1]\di \R$ with an oscillation function, the set $C_{f}\ne \emptyset$ \(or: is dense\).\label{pono0}
\item For Riemann integrable usco $f:[0,1]\di \R$, the set $C_{f}\ne \emptyset$ \(or: is dense\).\label{pono1}
\item For Riemann integrable $f:[0,1]\di \R$ with an oscillation function and $\int_{0}^{1}f(x)dx=0$, there is $x\in [0,1]$ with $f(x)=0$. \(Bourbaki, \cite{boereng}*{p.\ 61}\).\label{pono2}
\item For Riemann integrable usco $f:[0,1]\di [0,1]$ with $\int_{0}^{1}f(x)dx=0$, there is $x\in [0,1]$ with $f(x)=0$. \(Bourbaki, \cite{boereng}*{p.\ 61}\). \label{pono3}
\item \textsf{\textup{(FTC)}} For Riemann integrable $f:[0,1]\di \R$ with an oscillation function and $F(x):=\lambda x.\int_{0}^{x}f(t)dt$, there is $x_{0}\in (0,1)$ where $F(x)$ is differentiable with derivative $f(x_{0})$.\label{ponofar}
\item \textsf{\textup{(FTC)}} For Riemann integrable usco $f:[0,1]\di \R$ with $F(x):=\lambda x.\int_{0}^{x}f(t)dt$, there is $x_{0}\in (0,1)$ where $F(x)$ is differentiable with derivative $f(x_{0})$.\label{ponofar2}
\item For Riemann integrable usco $f:[0,1]\di \R^{+}$, there are $a, b\in  [0,1]$ with $0<\inf_{x\in [a,b]}f(x)$.\label{ponofaro}
\item For Riemann integrable $f:[0,1]\di \R^{+}$ with an oscillation function, there are $a, b\in  [0,1]$ with $0<\inf_{x\in [a,b]}f(x)$.\label{ponofaro2}
\item For totally discontinuous $f:[0,1]\di \R$ with oscillation function $\osc_{f}:[0, 1]\di \R$, there is $N\in \N$ and $E\subset  [0,1]$ of positive measure with $\osc_{f}(x)\geq \frac{1}{2^{N}}$ for $x\in E$.\label{topanga}
\item \emph{Volterra's theorem}: there do not exist two Riemann integrable $f, g:[0,1]\di \R$ with associated oscillation functions for which the continuity points of one are the discontinuity points of the other, and vice versa.\label{topanga4}
\item \emph{Volterra's corollary}: there is no Riemann integrable function with an oscillation function that is continuous on $\Q\cap[0,1]$ and discontinuous on $[0,1]\setminus\Q$.\label{topanga3}
\end{enumerate}
We can replace `usco' by `cliquish with an oscillation function' in the above.
We can restrict to pointwise discontinuous functions in items \eqref{topanga4}-\eqref{topanga3}. 
\end{thm}
\begin{proof}
First of all, assume $\WBCT_{[0,1]}$ and let $f:[0,1]\di \R$ be Riemann integrable with oscillation function $\osc_{f}:[0,1]\di \R$.  
Suppose $[0,1]=D_{f}=\cup_{k\in \N}D_{k}$ where $D_{k}:=\{x\in [0,1]:\osc_{f}(x)\geq \frac{1}{2^{k}}  \}$, which is closed since any oscillation function is usco by Theorem \ref{flappie}.
By $\WBCT_{[0,1]}$, there is $k_{0}\in \N$ such that $D_{k_{0}}$ does not have measure zero. 
Let $\eps_{0}>0$ be such that the measure of $D_{k_{0}}$ is at least $\eps_{0}$ and let $\delta_{0}>0$ be as provided by the definition of Riemann integrability for $\eps=\eps_{0}\frac{1}{2^{k_{0}+1}}$.   
For any partition $P$ with mesh below $\delta_{0}$  given by $0=x_{0}, t_{0},x_{1}, t_{1}, x_{1}, \dots, x_{k-1}, t_{k}, x_{k}=1 $, we define two partitions $P', P''$ as follows: in case $D_{k_{0}}\cap [x_{i}, x_{i+1}]= \emptyset$, do nothing; otherwise, change $t_{i}$ from $P$ to some $t_{i}'$ for $P'$ and $t_{i}''$ for $P''$ such that $|f(t_{i}')-f(t_{i}'')|> \frac{1}{2^{k_{0}+1}}$. The latter reals exist by the assumption $D_{k_{0}}\cap [x_{i}, x_{i+1}]\ne\emptyset$.  One readily shows that $|S(f, P')-S(f, P'')|>\eps_{0}\frac{1}{2^{k_{0}+1}}$, i.e.\ a contradiction, and item \eqref{pono0} follows. 
Note that the finite manipulation of $P$ can be done using $\IND_{\R}$.   Item \eqref{topanga3} follows in the same way by assuming $[0,1]\setminus \Q=D_{f} $; the proof then proceeds in the same way 
by noting $[0,1]=\cup_{k\in \N}E_{k}$ where $E_{k}:=D_{k}\cup\{q_{k}\}$ is closed and where $(q_{n})_{n\in \in \N}$ is an enumeration of the rationals in $[0,1]$.
Noting that Thomae's function from \eqref{thomae} is Riemann integrable, item \eqref{topanga4} also follows. 

\smallskip

Secondly, assume item \eqref{pono0} and let $(O_{n})_{n\in \N}$ be a sequence as in $\WBCT_{[0,1]}$.    Define $X_{n}:=[0,1]\setminus O_{n}$ and consider $h:[0,1]\di \R$ from \eqref{mopi}, which we show to be Riemann integrable with zero integral. 
To this end, fix $\eps>0$ and let $n_{0}\in \N$ be such that $\frac{1}{2^{n_{0}+1}}<\eps$.  Since $X_{n_{0}}$ has measure zero, we can find a sequence $(I_{n})_{n\in \N}$ of intervals that covers $X_{n_{0}}$ and has total length at most $\frac{1}{2^{n_{0}+2}}$.  By \cite{dagsamVII}*{Theorem 3.3}, the closed set $X_{n_{0}}$ is also covered by $\cup_{n\leq k_{0}}I_{n}$ for some $k_{0}\in \N$.    
Now let $P$ be any partition with mesh at most $\frac{1}{2^{n_{0}+2}}$.
To show that $\eps>S(f, P)=\sum_{i=0}^{|P|}h(t_{i}) (x_{i+1}-x_{i})  $, we split the intervals $[x_{i},x_{i+1}]$ into sub-intervals $[x_{i+1}, a_{0}], [a_{0}, a_{1}], \dots, [a_{m}, a_{m+1}], [a_{m+1}, x_{i}]$, where the $a_{i}$ are end-points of the $I_{n}$ in $[x_{i+1}-x_{i}]$ for some $n\leq k_{0}$.  
Now split $S(f, p)$ into two sums, one over the intervals not in $\cup_{n\leq k_{0}}I_{n}$, and one with the rest. 
By our assumptions, each separate sum is at most $\frac{1}{2^{n_{0}+2}}$ as required.  Now apply item \eqref{pono0} for $h$ (see Theorem~\ref{fronkn} for $\osc_{h}:[0,1]\di \R$) and note that $y\in C_{h}$ implies $h(y)=\osc_{h}(y)=0$, i.e.\ $y\in \cap_{n\in \N}O_{n}$, as required for $\WBCT_{[0,1]}$.  Similarly, item \eqref{topanga3} implies $\WBCT_{[0,1]}$ by considering \eqref{mopi} where $(X_{n})_{n\in \N}$ is a sequence of closed sets with $[0,1]\setminus \Q=\cup_{n\in \N}X_{n}$.  Moreover, item \eqref{topanga4} implies item \eqref{topanga3} in light of Thomae's function from \eqref{thomae},  

\smallskip

Thirdly, for item \eqref{pono1}, the absence of an oscillation function complicates matters (slightly).  We modify $F_{q}$ defined as in \eqref{ling} as follows: let $`x\in E_{q, l}$' in case
\[\textstyle
f(x)\leq q \wedge (\forall N\in \N)(\exists z\in B(x, \frac{1}{2^{N}}))( f(z)>q+\frac{1}{2^{l}}  ).
\]
As in the proof of Theorem \ref{clockn}, $E_{q,l}$ is closed and $D_{f}\subseteq\cup_{l\in \N, q\in \Q}E_{q,l}$. 
Now suppose $C_{f}=\emptyset$, i.e.\ $\cup_{l\in \N, q\in \Q}E_{q,l}=[0,1]$.  
By $\WBCT_{[0,1]}$, there is $l_{0}\in \N$, $q_{0}\in \Q$ such that $E_{q_{0},l_{0}}$ has non-zero measure. 
As in the first paragraph of the proof, $f$ is not Riemann integrable, and item \eqref{pono1} follows.  The latter implies $\WBCT_{[0,1]}$ in the same way as in the previous paragraph, noting that $h$ from \eqref{mopi} is usco by Theorem \ref{flap}.  

\smallskip

Fourth, for items \eqref{pono2} and \eqref{pono3}, the latter follow from items \eqref{pono0} and items \eqref{pono1} by noting that $x\in C_{f}$ and $\int_{0}^{1}f(x)dx=0$ implies $f(x)=0$.
For the reversals, use $h:[0,1]\di \R$ from \eqref{mopi}, which has Riemann integral equal to zero by the second paragraph of this proof.    

\smallskip

Fifth, for items \eqref{ponofar} and \eqref{ponofar2}, the latter are equivalent to items \eqref{pono2} and \eqref{pono3} in the same way.
In particular, the function $h:[0,1]\di \R$ from \eqref{mopi} has a primitive, namely the constant zero function.  Moreover, the usual epsilon-delta proof establishes that $F(x):=\int_{0}^{x}f(t)dt$ is continuous and that $F'(y)=f(y)$ if $f$ is continuous at $y$.  

\smallskip

Sixth, item \eqref{ponofaro} immediately follows from item \eqref{pono0} by taking a small enough ball around a point of continuity.  
Similarly, item \eqref{ponofaro} implies item \eqref{pono3} by noting that $\int_{a}^{b}f(x)dx\leq \int_{c}^{d}f(x)dx$ in case $f$ is non-negative and $[a, b]\subset [c, d]$.
Item \eqref{ponofaro2} is treated in the same way.
The final sentence follows by Theorem \ref{flap}. 

\smallskip

Seventh, let $f:[0,1]\di \R$ be as in item \eqref{topanga} and define the closed set $D_{k}:=\{ x\in [0,1]: \osc_{f}(x)\geq\frac{1}{2^{k}}\}$. 
By assumption, $[0,1]=\cup_{k\in \N}D_{k}$ and $\WBCT_{[0,1]}$ implies there is $k_{0}\in\N$ and $E\subset [0,1]$ of positive measure such that $E\subset D_{k_{0}}$, as required.
For the reversal, assume item \eqref{topanga} and let $f:[0,1]\di \R$ be Riemann integrable and with oscillation function $\osc_{f}:[0,1]\di \R$.  
In case $C_{f}\ne \emptyset$, we are done in light of item \eqref{pono1}. 
In case $C_{f}=\emptyset$, $f$ is totally discontinuous and we can apply item~\eqref{topanga}.  
Let $N_{0}\in \N$ and $E\subset  [0,1]$ of positive measure be such that $\osc_{f}(x)\geq \frac{1}{2^{N_{0}}}$ for $x\in E$.  
As in the first paragraph of this proof, $f$ is not Riemann integrable, a contradiction.  
For the penultimate sentence, it suffices to observe that $h$ from \eqref{mopi} is cliquish and has an oscillation function by Theorems \ref{fronkn} and \ref{flap}.
The final sentence follows by observing that the functions used in the above proof of the implications $\eqref{topanga4}\di \eqref{topanga3}\di\WBCT_{[0,1]}$ are pointwise discontinuous.  
\end{proof}
In the previous theorem, we can replace the use of $\QFAC^{0,1}$ (not provable in $\ZF$) by $\NCC$ (provable in $\Z_{2}^{\Omega}$) from \cite{dagsamIX}.
It seems we can (also) avoid the use of $\QFAC^{0,1}$ if we require the sets to have \emph{Jordan} measure zero. 
We do not know if we can weaken the induction axiom in Theorem \ref{duck5}.   
We do have the following corollary, which is of separate interest.
\begin{cor}[$\ACAo$]\label{sepa}
For Riemann integrable $f:[0,1]\di \R$ with oscillation function $\osc_{f}:[0,1]\di \R$, $D_{k}:=\{x\in [0,1]:\osc_{f}(x)\geq \frac{1}{2^{k}}\} $ has measure zero.
\end{cor}
\begin{proof}
The first paragraph of the proof of the theorem contains the required proof by contradiction, starting with `\emph{Let $\eps_{0}>0$ be such that\dots}'.
\end{proof}
Corollary \ref{sepa} identifies the `missing link' to the Vitali-Lebesgue theorem:  each $D_{k}$ has measure zero and since $D_{f}=\cup_{k\in\N}D_{k}$, 
we only need Tao's pigeonhole principle for measure spaces to conclude that $D_{f}$ has measure zero.  

\smallskip

Finally, with the gift of hindsight, one can show that the following strengthening of $\WBCT_{[0,1]}$ is equivalent to $\BCT_{[0,1]}$ again.  
\begin{center}
\emph{For a sequence $(C_{n})_{n\in \N}$ of closed sets with $[0,1]=\cup_{n\in \N}C_{n}$, there is $m\in \N$ such that $C_{m}$ has non-zero measure and 
$[0,1]$ cannot be covered by finite copies of $C_{m}$.}  
\end{center}
We could also obtain an equivalence between $\WBCT_{[0,1]}$ and a version of the uniform boundedness principle from Theorem \ref{rink}.  
Moreover, $\WBCT_{[0,1]}$ implies the following interesting statement, which follows from e.g.\ \cite{bage}*{Theorem}, 
\begin{center}
\emph{for a closed set $E\subset \R$ of measure zero, there is $x\in \R$ such that $x+E\subset \R\setminus \Q$,}
\end{center}
and we expect a reversal to be possible.

\subsection{The pigeonhole principle}\label{duif}
We introduce Tao's {pigeonhole principle for measure spaces} from \cite{taoeps} and obtain equivalences involving the fundamental theorem of calculus and the Vitali-Lebesgue theorem.
Hence, the latter theorems are not provable in $\Z_{2}^{\omega}$ while the restrictions to Baire 1 and quasi-continuous functions and RM-closed sets are provable in $\ACAo$ by Theorems \ref{pinko} and \ref{qvl}. 

\smallskip

First of all, we shall study the pigeonhole principle for \emph{closed} sets as in $\PHP_{[0,1]}$ below.  
We recall that like in second-order RM, statements of the form `$E\subset [0,1]$ has measure $0$' can be made without introducing the Lebesgue measure.
\begin{princ}[$\PHP_{[0,1]}$]\label{PHP}
If $ (X_n)_{n \in \N}$ is an increasing sequence of measure zero closed sets of reals in $[0,1]$, then $ \bigcup_{n \in\N } X_n$ is measure zero.
\end{princ} 
By the the main result of \cite{trohim}, not all nowhere dense measure zero sets are the countable union of measure zero closed sets, i.e.\ $\PHP_{[0,1]}$ does not generate `all' measure zero sets. 
With some effort, one can derive $\PHP_{[0,1]}$ restricted to R2-closed sets from Cousin's lemma as studied in \cite{dagsamIII}.

\smallskip

Secondly, we establish that while $\PHP_{[0,1]}$ is hard to prove in terms of conventional comprehension, the former principle can be proved \emph{without} the Axiom of Choice, i.e.\ in a fragment of $\ZF$.
Thus, the properties of $\PHP_{[0,1]}$ are not due to the latter implying (a fragment of) the Axiom of Choice.  By contrast, $\PHP_{[0,1]}$ \emph{restricted to RM-codes} is provable in a relatively weak system.  
\begin{thm}\label{pinko}~
\begin{itemize}
\item The $\Z_{2}^{\omega}+\QFAC^{0,1}$ does not prove $\PHP_{[0,1]}$.
\item The system $\Z_{2}^{\Omega}$ proves $\PHP_{[0,1]}$.  
\item The system $\ACAo$ proves $\PHP_{[0,1]}$ restricted to RM-codes for closed sets. 
\end{itemize}
\end{thm}
\begin{proof}
For the first item, we have $\PHP_{[0,1]}\di \NIN_{[0,1]}$ as for any injection $Y:[0,1]\di \N$, the set $E_{n}:=\{x\in [0,1]:Y(x)=n\}$ has measure zero but the union does not.  
By \cite{dagsamXI}*{Theorem 3.1}, the system  $\Z_{2}^{\omega}+\QFAC^{0,1}$ cannot prove $\NIN_{[0,1]}$. 

\smallskip

For the second item, the following formula expresses that $(X_{n})_{n\in \N}$ is a sequence of closed measure zero sets. 
\be\label{tempa}\textstyle
(\forall n, k\in \N)( \exists (I_{m})_{m\in \N}  )\big( X_{n}\subset \cup_{i\in \N} I_{i} \wedge \frac{1}{2^{k}}>  \sum_{j\in \N} |I_{j}|\big), 
\ee
where each $I_{n}$ is an open interval with rational end-points.  
By \cite{dagsamIX}*{Theorems 3.1 and 3.5}, $\Z_{2}^{\Omega}$ proves that countable coverings of closed sets have finite sub-coverings.  Applying this covering result to \eqref{tempa}, we obtain the following formula
\[\textstyle
(\forall n, k\in \N)( \exists q_{0},\dots , q_{2k} \in \Q  )\big[ X_{n}\subset \cup_{1\leq i\leq k} (q_{2i-1}, q_{2i}) \wedge \frac{1}{2^{k}}>  \sum_{1\leq i\leq k} |q_{2i-1}- q_{2i}|\big],
\]
where the formula in square brackets is decidable using $\exists^{3}$.   Now apply $\QFAC^{0,0}$, provable in $\ZF$ and included in $\RCAo$, to obtain the conclusion of $\PHP_{[0,1]}$.

\smallskip

For the third item, consider \eqref{tempa} and assume $(X_{n})_{n\in \N}$ is given as a sequence $(C_{n})_{n\in \N}$ of RM-closed sets.
By \cite{brownphd}*{Lemma~3.13}, a countable covering of an RM-closed set has a finite sub-covering assuming $\WKL_{0}$, i.e.\ we obtain 
\[\textstyle
(\forall n, k\in \N)( \exists q_{0},\dots , q_{2k} \in \Q  )\big[ C_{n}\subset \cup_{1\leq i\leq k} (q_{2i-1}, q_{2i}) \wedge \frac{1}{2^{k}}>  \sum_{1\leq i\leq k} |q_{2i-1}- q_{2i}|\big].
\]
In case the formula in square brackets is decidable (using $\exists^{2}$), we can apply $\QFAC^{0,0}$ and obtain the required interval covering of $\cup_{n\in \N}C_{n}$.  

\smallskip

To show that $C_{n}\subset \cup_{1\leq i\leq k} (q_{2i-1}, q_{2i})$ is decidable using $\exists^{2}$, note that this formula is equivalent to `$(\forall x\in [0,1])(x\in O_{n,k})$', where the RM-open set $O_{n,k}\subset [0,1]$ is the union of the RM-open sets $[0,1]\setminus C_{n}$ and $\cup_{1\leq i\leq k} (q_{2i-1}, q_{2i})$.  By \cite{simpson2}*{II.7.1},  there is a (code for a) continuous function $f_{n,k}$ such that $x\in O_{n,k}\asa f_{n,k}(x)>0$ for all $n,k\in \N$ and $x\in [0,1]$.
By \cite{dagsamXIV}*{Theorem~2.2}, codes for continuous functions on $\R$ yield third-order functions on $\R$, even in $\RCAo$. 
Using $\exists^{2}$, we may compute the infimum of $f_{n,k}$ on $[0,1]$ by \cite{kohlenbach2}*{Prop.~3.14}. 
The following equivalence now holds for all $n,k\in \N$:
\be\label{Deci}
\big[C_{n}\subset \cup_{1\leq i\leq k} (q_{2i-1}, q_{2i})\big]\asa \big[0< \inf_{x\in [0,1]}f_{n,k}(x)\big],
\ee
as continuous functions attain their infimum on the unit interval given $\WKL_{0}$ by \cite{simpson2}*{IV.2.3}.
The right-hand side of \eqref{Deci} is decidable given $\exists^{2}$. 
 \end{proof}
In light of the previous proof, $\PHP_{[0,1]}$ follows from the principle $\MCC$ from \cite{dagsamIX}.  Moreover, the third item in Theorem \ref{pinko} goes through for `$\ACAo$' replaced by `$\RCAo+\WKL$', using \cite{kohlenbach2}*{Prop.\ 3.15} to obtain \eqref{Deci} without recourse to $(\exists^{2})$.

\smallskip

\noindent
Thirdly, the following theorem should be contrasted with Theorem \ref{duck555}, keeping in mind the close connection between quasi-continuity and cliquishness (see Remark~\ref{donola}); the set $C_{f}$ exists by Theorems \ref{flang} and \ref{plonkook}.
\begin{thm}[$\ACAo$]\label{qvl}~
\begin{itemize}
\item For Riemann integrable quasi-contiuous $f:[0,1]\di \R$, $C_{f}$ has measure one.
\item For Riemann integrable Baire 1 $f:[0,1]\di \R$, the set $C_{f}$ has measure one.
\end{itemize}
\end{thm}
\begin{proof}
For the first item, let $f:[0,1]\di \R$ be quasi-continuous.
Our starting point is the proof of Theorem \ref{plonkook}, in particular the set $O_{m}$ defined as the collection of all $x\in [0,1]$ such that \eqref{obvio}.
By the definition of quasi-continuity, \eqref{obvio} and \eqref{poiuy} are equivalent, and this remains true if $N=M$.  
Using $\mu^{2}$, let $Y_{m}(x)$ be the least $N\in \N$ as in \eqref{obvio}, if such exists, and zero otherwise.  
Hence, we $x\in O_{m}\asa x\in \cup_{q\in \Q\cap O_{m}}B(q,\frac{1}{2^{Y_{m}(q)}})$, for any $x\in [0,1]$. 
In this way, we have obtained an RM-code for the sequence $(O_{m})_{m\in \N}$.  Recall that $D_{f}=\cup_{m\in \N}D_{m}$ where $D_{m}:=[0,1]\setminus O_{m}$ and $D_{m}:=\{x\in [0,1]:\osc_{f}(x)\geq \frac{1}{2^{m}}\}$
If $f:[0,1]\di \R$ is also Riemann integrable then Corollary \ref{sepa} guarantees that the set $D_{k}:=\{x\in [0,1]:\osc_{f}(x)\geq \frac{1}{2^{k}}\}$ has measure zero.  
Now apply Theorem \ref{pinko} to conclude that $D_{f}$ has measure zero.  

\smallskip

For the second item, let $f:[0,1]\di \R$ be the pointwise limit of the sequence of continuous functions $(f_{n})_{n\in \N}$.
Now consider the sets $F_{m,n}$ from the proof of Theorem \ref{flang}.
One readily proves that $\cup_{m\in \N} F_{m,n} \subset D_{n}:=\{x\in [0,1]:\osc_{f}(x)\geq \frac{1}{2^{n}}\}$.  The latter is measure zero by Corollary \ref{sepa}, in case $f$ is also Riemann integrable.
By the proof of Theorem \ref{flang}, $F_{n,m}$ is represented by RM-codes, i.e.\ we may apply the third item of Theorem \ref{pinko} to conclude that $\cup_{n,m\in \N} F_{m,n}$ has measure zero.
Since $D_{f}$ is a subset of the latter union, we are done. 
\end{proof}
Fourth, we have the following theorem where we note that the sets $C_{f}$ and $ D_{f}$ exists by Theorem \ref{plonk} in e.g.\ item \eqref{tao2} of Theorem \ref{duck555}. 
The last sentence of the theorem should be contrasted with Corollary \ref{qvl}, recalling Remark \ref{donola}.
We recall the induction fragment $\IND_{\R}$ introduced in Section \ref{1227}.
\begin{thm}[$\ACAo+\IND_{\R}+\QFAC^{0,1}$]\label{duck555}
The following are equivalent.  
\begin{enumerate}
\renewcommand{\theenumi}{\alph{enumi}}
\item The pigeonhole principle for measure spaces as in $\PHP_{[0,1]}$.
\item \(Vitali-Lebesgue\) For Riemann integrable $f:[0,1]\di \R$ with an oscillation function, the set $D_{f}$ has measure $0$.\label{tao1}
\item For Riemann integrable usco $f:[0,1]\di \R$, the set $D_{f}$ has measure $0$.\label{tao2}
\item For Riemann integrable $f:[0,1]\di [0,1]$ with an oscillation function and $\int_{0}^{1}f(x)dx=0$, the set $\{x\in [0,1]:f(x)=0\}$ has measure one.\label{tao3}
\item For Riemann integrable usco $f:[0,1]\di [0,1]$ with $\int_{0}^{1}f(x)dx=0$, the set $\{x\in [0,1]:f(x)=0\}$ has measure one.\label{tao4}
\item \textsf{\textup{(FTC)}} For Riemann integrable $f:[0,1]\di \R$ with an oscillation function and $F(x):=\lambda x.\int_{0}^{x}f(t)dt$, the following set exists: 
\be\label{tanko}
\{x\in [0,1]:F \textup{ is differentiable at $x$ with derivative } f(x)\}
\ee
and has measure one.\label{tao5}
\item \textsf{\textup{(FTC)}} The previous item for usco functions.\label{tao6}
\item For $f:[0,1]\di \R$ not continuous almost everywhere with oscillation function $\osc_{f}:[0, 1]\di \R$, 
there is $N\in \N$ and $E\subset  [0,1]$ of positive measure such that $\osc_{f}(x)\geq \frac{1}{2^{N}}$ for $x\in E$.\label{topanga2}
\end{enumerate}
We can replace `usco' by `cliquish with an oscillation function' in the above.  
\end{thm}
\begin{proof}
First of all, most items follows as in the proof of Theorem \ref{duck5}.  For instance, assume $\PHP_{[0,1]}$ and suppose for $f:[0,1]\di \R$ as in item \eqref{tao1} of the theorem, $D_{f}$ has positive measure.  
Then some $D_{k}:=\{x\in [0,1]:\osc_{f}(x)\geq \frac{1}{2^{k}}\}$ has positive measure by $\PHP_{[0,1]}$.  Now proceed as in the first paragraph of the proof of Theorem~\ref{duck5} to derive a contradiction. 

\smallskip

In turn, item \eqref{tao1} of the theorem implies $\PHP_{[0,1]}$ by fixing a sequence $(X_{n})_{n\in \N}$ of closed sets and considering the second paragraph of the proof of Theorem \ref{duck5}; the only modification is that $C_{h}$ has measure one, which implies that $\cup_{n\in \N}X_{n}$ has measure zero (instead of just being a strict subset), and $\PHP_{[0,1]}$ follows as required for the equivalence between the first two items.
The equivalences for items \eqref{tao2}-\eqref{tao4} are proved in the same way.  

\smallskip

For items \eqref{tao5} and \eqref{tao6}, the usual epsilon-delta proof establishes that $F(x):=\int_{0}^{x}f(t)dt$ is continuous and that $F'(y)=f(y)$ if and only if $f$ is continuous at $y$.  
Hence, the set $C_{f}$ is exactly the set in \eqref{tanko}.

\smallskip

Finally, let $f:[0,1]\di \R$ be as in item \eqref{topanga2} and define the closed set $D_{k}:=\{ x\in [0,1]: \osc_{f}(x)\geq\frac{1}{2^{k}}\}$. 
By assumption, $D_{f}=\cup_{k\in \N}D_{k}$ has positive measure and $\PHP_{[0,1]}$ implies there is $k_{0}\in\N$ such that $ D_{k_{0}}$ has positive measure, as required.
For the reversal, assume item \eqref{topanga2} and let $f:[0,1]\di \R$ be Riemann integrable oscillation function $\osc_{f}:[0,1]\di \R$.  
In case $C_{f}$ has measure one, we are done in light of item \eqref{tao1}. 
In case $C_{f}$ does not have measure one, $f$ is not continuous almost everywhere and we can apply item \eqref{topanga2}.  
Let $N_{0}\in \N$ and $E\subset  [0,1]$ of positive measure be such that $\osc_{f}(x)\geq \frac{1}{2^{N_{0}}}$ for $x\in E$.  
As in the first paragraph of this proof, $f$ is not Riemann integrable, a contradiction, and we are done. 
\end{proof}
We note that $\PHP_{[0,1]}$ does not (seem to) provide a proof of the fact that sequences of measure zero sets have RM-codes.  
Similar to the third item of Theorem~\ref{deng}, we \emph{could} study $\PHP_{[0,1]}$ restricted to sequences $(X_{n})_{n\in \N}$ such that $[0,1]\setminus \cup_{n\in \N}X_{n}$ is dense.
By Theorem \ref{duck5}, this study would however require either $\WBCT_{[0,1]}$ or the restriction to Riemann integrable functions with dense $C_{f}$.  
This suggests an asymmetry between measure and category from the RM-point of view. 

\smallskip

Thus, we have established equivalences involving $\PHP_{[0,1]}$ and the Vitali-Lebesgue theorem, rendering the latter unprovable in $\Z_{2}^{\omega}$.
The same holds for the restrictions to usco or cliquish functions while the restrictions to quasi-continuous functions are provable in $\ACAo$.
In contrast to the close relation between cliquishness and quasi-continuity, there is a great divide, some might say `abyss', between $\ACAo$ and $\Z_{2}^{\omega}$. 
We have no explanation for this phenomenon at the moment.  

\smallskip

Finally, the following theorem can be proved using the Baire category theorem.
\begin{center}
\emph{For any continuous $f:[0,1] \di \R$, define $f_{0}=f$ and $f_{k+1}(x):=\int_{0}^{x}f_{k}(t)dt$.  In case $(\forall x\in [0,1])(\exists k\in \N)(f_{k}(x)=0 )$, then $f$ is identically zero.}
\end{center}
Replacing `continuous $f:[0,1] \di \R$' by `Riemann integrable $f:[0,1] \di [0, 1]$' and `identically' by `almost everywhere', the resulting theorem is equivalent to $\PHP_{[0,1]}$. 
We conjecture there to be many more similar variations.  

\section{A Bigger Picture}\label{bigger}
\noindent
In this section, we discuss the bigger picture associated with our above results.
\begin{itemize}
\item In Section \ref{reflm}, we discuss the intimate connection between our above results and the RM of the uncountability of $\R$ (\cite{samBIG}). 
\item In Section \ref{biggerer}, we exhibit parallels between some of our results and well-known properties of the Axiom of Choice in set theory. 
\item In Section \ref{nex}, we discuss Banach's theorem from \cite{kanach} as follows: 
\begin{center}
\emph{the continuous nowhere differentiable functions are dense in $C([0,1])$}.
\end{center}
We show that \emph{slight} variations of Banach's result are not provable in $\Z_{2}^{\omega}$.  
\end{itemize}
Regarding the final item, we shall study the notion of \emph{generalised absolute continuity}, which is intermediate between the continuous and absolutely continuous functions. 
We refer to e.g.\ \cites{bartle,gordonkordon, saks} for details but do mention that generalised absolute continuity is intimately connected to integration theory and the (most general version of) the fundamental theorem of calculus. 

\smallskip

Finally, the reader should recall Remark \ref{dichtbij} before reading Section \ref{nex}.  In the former remark, it is shown that 
closed sets as in Definition \ref{char}, i.e.\ without any additional representation, are readily found in basic analysis, like the study 
of regulated or bounded variation functions.    The above study of usco functions provides similar, but less basic, examples. 
The results in Section \ref{nex} are the most basic in nature, as they pertain to continuous functions.  The implications for coding in second-order RM are clearly significant.  

\subsection{Reflections on and of Reverse Mathematics}\label{reflm}
We discuss Figure \ref{xxx} below which summarises how the above principles are related; we also discuss how our results are intimately connected to the RM of the uncountability of $\R$ from \cite{samBIG}.

\smallskip

First of all, we note that the statement \emph{a countable set is meagre} has a trivial proof in $\RCAo$, i.e.\ there is some non-trivial asymmetry between measure and category in Figure \ref{xxx}. 
Furthermore, we {conjecture} that $\BCT_{[0,1]}$ does not imply $\PHP_{[0,1]}$, say over $\Z_{2}^{\omega}+\QFAC^{0,1}$, and vice versa.  
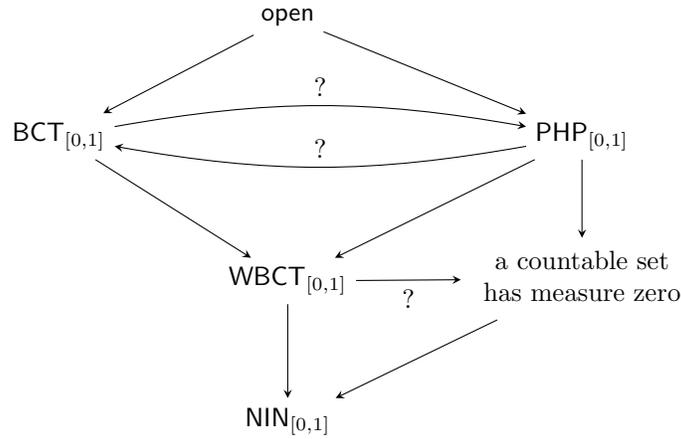
\begin{figure}[H]
\begin{tikzpicture}
  \matrix (m) [matrix of math nodes,row sep=3em,column sep=4em,minimum width=2em]
  { ~ & \open& ~\\
     \BCT_{[0,1]} & ~              & \PHP_{[0,1]}\\
       ~          &  \WBCT_{[0,1]} &           \textup{\begin{tabular}{c}a countable set \\ has measure zero\end{tabular}} \\
        ~          &  \NIN_{[0,1]} &           ~ \\};
  \path[-stealth]
    (m-1-2) edge node [left] {} (m-2-1)
        (m-1-2) edge node [left] {} (m-2-3)
    (m-2-1) edge node [left] {} (m-3-2)
    (m-2-3) edge node [left] {} (m-3-2)
    (m-3-2) edge node [left] {} (m-4-2)
    (m-2-3) edge node [left] {} (m-3-3) 
    (m-3-3) edge node [left] {} (m-4-2)
    (m-3-2) edge node [below] {?} (m-3-3)
    (m-2-1) edge[bend left=10] node [above] {?} (m-2-3)
    (m-2-3) edge[bend right=-10] node [above] {?} (m-2-1)
  ;
\end{tikzpicture}
\caption{Some relations among our principles}
\label{xxx}
\end{figure}

Secondly, we have developed the RM of the uncountability of $\R$ in \cite{samBIG}. The following where shown to be equivalent over $\ACAo+\QFAC^{0,1}+\FIN$, where the axiom $\FIN$ allows one to neatly handle finite sets of reals.
\begin{enumerate}
 \renewcommand{\theenumi}{\alph{enumi}}
\item The principle $\NIN_{\alt}$: for a sequence of \textbf{finite} sets $(X_{n})_{n\in \N}$ in $[0,1]$, there is $x\in [0,1]\setminus \cup_{n\in \N}X_{n}$. 
\item For regulated $f:[0,1]\di \R$, there is a point $x\in [0,1]$ where $f$ is continuous (or quasi-continuous, or lower semi-continuous, or Darboux).\label{fila1}
\item For regulated $f:[0,1]\di [0,1]$ with Riemann integral $\int_{0}^{1}f(x)dx=0$, there is $x\in [0,1]$ with $f(x)=0$ (Bourbaki, \cite{boereng}*{p.\ 61, Cor.\ 1}).
\item (Volterra \cite{volaarde2}) For regulated $f,g:[0,1]\di \R$, there is $x\in [0,1]$ such that $f$ and $g$ are both continuous or both discontinuous at $x$. \label{volkerta1}
\item (Volterra \cite{volaarde2}) For regulated $f:[0,1]\di \R$, there is either $q\in \Q\cap [0,1]$ where $f$ is discontinuous, or $x\in [0,1]\setminus \Q$ where $f$ is continuous.\label{volkerta2}
\item \textsf{(FTC)} For regulated $f:[0,1]\di \R$, there is $y\in (0,1)$ where $F(x):=\lambda x.\int_{0}^{x}f(t)dt$ is differentiable with derivative equal to $f(y)$.\label{bongk}
\item Blumberg's theorem (\cite{bloemeken}) restricted to regulated functions. \label{fila8}
\end{enumerate}
Clearly item \eqref{bongk} involves the fundamental theorem of calculus while item \eqref{volkerta1} involves Volterra's early work, all restricted to \emph{regulated} functions.  
We have obtained similar results for functions of bounded variation.  
In this paper, we have shown that item~\eqref{bongk} for usco functions is equivalent to $\BCT_{[0,1]}$ (Theorem \ref{flonk}) while item \eqref{volkerta1} for usco functions is equivalent to $\PHP_{[0,1]}$ (Theorem \ref{duck555}).  We have obtained similar results for cliquish functions (that come with an oscillation function).    

\smallskip

In light of the above, we can say that the RM of the uncountability of $\R$ is the product of taking the RM of $\PHP_{[0,1]}$ and the RM of $\BCT_{[0,1]}$ and pushing everything down from usco/cliquish functions to the level of regulated/bounded variation functions.  This observation provides a partial explanation why the RM of $\PHP_{[0,1]}$ and the RM of $\BCT_{[0,1]}$ can proceed along the same lines despite the apparent differences between measure and category.   

\smallskip

Thirdly, with the gift on hindsight, we could obtain equivalences between the following principles, say working over the aforementioned system $\ACAo+\QFAC^{0,1}+\FIN$ from \cite{samBIG}.
The first item expresses that \emph{countable sets have measure zero}, as can be found in Figure \ref{xxx} on the right-hand side. 
\begin{itemize} 
\item For a sequence of \textbf{finite} sets $(X_{n})_{n\in \N}$ in $[0,1]$, $\cup_{n\in \N}X_{n}$ has measure zero.
\item For regulated $f:[0,1]\di \R$, the set $C_{f}$ has measure one. 
\item The previous item for `continuity' replaced by quasi-continuity or usco.  
\item For regulated $f:[0,1]\di [0,1]$ with Riemann integral $\int_{0}^{1}f(x)dx=0$, the set $\{ x\in [0,1]:f(x)=0\}$ has measure one.  
\item A set $A\subset \R$ without limit points has measure zero.  
\item \dots
\end{itemize}
Similar equivalences are possible for functions of bounded variation, like in \cite{samBIG}*{\S3.4}, and the many intermediate classes of `generalised' bounded variation (see \cite{voordedorst}). 
A real challenge would be to study \emph{Carleson's theorem} in this context (\cite{carleson1}). 

\smallskip

Finally, the above analogy is not absolute:  while \emph{symmetrically continuous} functions feature in the RM of the uncountability of $\R$ (see \cite{samBIG}*{Theorem 3.9}), we have not been able to obtain similar results for $\BCT_{[0,1]}$, $\PHP_{[0,1]}$, or $\WBCT_{[0,1]}$.

\subsection{Reflections on and of set theory}\label{biggerer}
We discuss the relationship between our results and set theory; the reader disinterested in foundational musings will skip this section.  In a nutshell, we will show evidence for a parallel between our results and some central results in set theory, like the mercurial nature of the cardinality of the real numbers, disasters in the absence of the Axiom of Choice ($\AC$ for short), and the essential role of $\AC$ in measure and integration theory.

\smallskip

First of all, following the famous work of G\"odel (\cite{goeset}) and Cohen (\cite{cohen1, cohen2}), the \emph{Continuum Hypothesis} cannot be proved or disproved in $\ZFC$, i.e.\ Zermelo-Fraenkel set theory with $\AC$, the usual foundations of mathematics.  Thus, the exact cardinality of $\R$ cannot be established in $\ZFC$.  
A parallel observation in higher-order RM is that $\Z_{2}^{\omega}$ cannot prove that there is no injection from $\R$ to $\N$ (see \cite{dagsamX} for details).  In a nutshell, the cardinality of $\R$ has a particularly mercurial nature, in both set theory and higher-order arithmetic.

\smallskip

Secondly, many standard results in mainstream mathematics are not provable in $\ZF$, i.e.\ $\ZFC$ with $\AC$ removed, as explored in great detail \cite{heerlijkheid}.
The absence of $\AC$ is even said to lead to \emph{disasters} in topology and analysis (see \cite{kermend}).  A parallel phenomenon was observed in \cites{dagsamVII, dagsamV}, namely 
that certain rather basic equivalences go through over $\RCAo+\QFAC^{0,1}$, but not over $\Z_{2}^{\omega}$.  

\smallskip

Examples include the equivalence between compactness results and local-global principles, which are intimately related according to Tao (\cite{taokejes}).  
In this light, it is fair to say that disasters happen in both set theory and higher-order arithmetic in the absence of $\AC$.
It should be noted that $\QFAC^{0,1}$ (not provable in $\ZF$) can be replaced by $\NCC$ from \cite{dagsamIX} (provable in $\Z_{2}^{\Omega}$) in the aforementioned. 

\smallskip

Thirdly, we discuss the essential role of $\AC$ in measure and integration theory, which leads to rather concrete parallel observations in higher-order arithmetic.
Indeed, the full pigeonhole principle for measure spaces is not provable in $\ZF$, which immediately follows from e.g.\ \cite{heerlijkheid}*{Diagram~3.4}.
A parallel phenomenon in higher-order arithmetic is that even the restriction to closed sets, namely $\PHP_{[0,1]}$ cannot be proved in $\Z_{2}^{\omega}+\QFAC^{0,1}$ by Theorem \ref{pinko} (but $\Z_{2}^{\Omega}$ suffices). 

\smallskip

Moreover, a more `down to earth' observation is obtained by comparing the topic of \cite{hahakaka} and item~\eqref{pono2} in Theorem~\ref{duck5}.  
Regarding \cite{hahakaka}, the authors show that $\AC$ is needed to justify our intuition of the Lebesgue integral as the area under the graph of a function.  
Theorem \ref{duck5} establishes the parallel observation that the same justification for the \emph{Riemann} integral cannot be proved in $\Z_{2}^{\omega}+\QFAC^{0,1}$ (but $\Z_{2}^{\Omega}$ suffices as usual).  

\subsection{On the scope of coding}\label{nex}
In the above, we have obtained equivalences for the Baire category theorem as in $\BCT_{[0,1]}$, which is not provable in $\Z_{2}^{\omega}$ by Theorem~\ref{tonko}.
By contrast, the restrictions of $\BCT_{[0,1]}$ to RM-codes and R2-representations are provable in $\ACAo$, following \cite{simpson2}*{II.5.8} and \cite{dagsamVII}*{Theorem 7.10}.
A natural question is whether the latter `coded' versions suffice for the study of $C([0,1])$, the space of continuous functions.   
In this section, we discuss this question and provide a counterexample based on \emph{generalised absolute continuity} as in Definition \ref{ACGDEF}.  

\smallskip

First of all, Banach (\cite{kanach}) already showed that `most' continuous functions on the unit interval are nowhere differentiable, in the sense that such functions are dense in $C([0,1])$.  
Nemoto (\cite{nemoto1}*{Theorem 5.4}) provides a proof of this fact in constructive mathematics, with the same coding for $C([0,1])$ and for open sets as in second-order RM (\cite{simpson2}*{II.5.6 and IV.2.13}).  
In particular, the proof makes use of the second-order Baire category theorem (\cite{simpson2}*{II.5.8} and \cite{nemoto1}*{\S3}) and goes through (up to insignificant technical details) in $\RCA_{0}$ with intuitionistic logic, following the results in \cite{nemoto2}.  By \cite{dagsamXIV}*{Cor.\ 2.5}, continuous functions on the unit interval have RM-codes, i.e.\ Banach's aforementioned theorem about $C([0,1])$ \emph{without codes} can be proved using the Baire category theorem for RM-codes, working in $\RCAo+\WKL$.  

\smallskip

Secondly, we now show that slight variations of the aforementioned theorem by Banach cannot be proved in $\Z_{2}^{\omega}$.  
To this end, we introduce some sub-classes of $C([0,1])$ in Definition \ref{ACGDEF} that are connected to integration theory.
In particular, $f:[0,1]\di \R$ is gauge (or: Perron, or: Denjoy) integrable if and only if there is $F\in \ACG^{*}$ such that $F'=f$ almost everywhere.  Details may be found in \cites{bartle,gordonkordon, saks}. 
\bdefi[Generalised absolute continuity]\label{ACGDEF} Fix $f:[0,1]\di \R$ and $X\subset [0,1]$.  
\begin{itemize}
\item We say that $f\in \AC(X)$ if for all $\eps>0$, there is $\delta>0$ such that for any finite collection of pairwise disjoint $(a_{0},b_{0}),(a_{1},b_{1}),É,(a_{k},b_{k})$ 
with endpoints in $X$ we have $\sum_{i\leq k}|a_{i}-b_{i}|<\delta\di \sum_{i\leq k}|f(a_{i})-f(b_{i})|<\eps$.
\item We say that $f\in \AC^{*}(X)$ if it satisfies the previous item with `$|f(a_{i})-f(b_{i})|$' replaced by `$\osc_{f}([a_{i}, b_{i}])$' in the conclusion.
\item We say that $f\in \ACG$ in case it is continuous on $[0,1]$ and there is a sequence of closed sets $(C_{n})_{n\in \N}$ such that $f_{\upharpoonright C_{n}}$ is $\AC(C_{n})$ for each $n\in \N$.
\item We say that $f\in \ACG^{*}$ for the previous item with `$\AC$' replaced by `$\AC^{*}$'.  
\end{itemize}
\edefi
Lee establishes the equivalence of various definitions of $\ACG^{*}$ in \cite{ACG}.  
The previous definition from \cite{ACG} is the most basic one, in our opinion.
By Remark~\ref{dichtbij}, closed sets as in Definition \ref{char}, i.e.\ without any representation, already pop up `in the wild' in the study of regulated and bounded variation functions.  
In this light, we should not assume closed sets to have extra constructive data, like RM-codes or R2-representations. 

\smallskip

Thirdly, we have the following theorem where we note that $\ACG$ and $\ACG^{*}$ functions are not nowhere differentiable (\cite{ACG}*{Theorem 4} and \cite{gordonkordon}*{p.\ 235}). 
In other words, the complements of $\ACG$ and $\ACG^{*}$ are larger than the class of nowhere differential functions from Banach's aforementioned theorem.  
\begin{thm}
The following cannot be proved in $\Z_{2}^{\omega}+\IND_{0}$.
\begin{itemize}
\item Any function in $\ACG$ is differentiable at some point of the unit interval.
\item Any function in $\ACG^{*}$ is differentiable at some point of the unit interval.
\item The complement of $\ACG$ is dense in $C([0,1])$.
\item The complement of $\ACG^{*}$ is dense in $C([0,1])$.
\end{itemize}
\end{thm}
\begin{proof}
First of all, a continuous and nowhere differentiable function is readily defined in $\RCA_{0}$ following the results in \cite{nemoto1}.  
One could also use the well-known definition of Weierstrass' function and establish its properties directly using \cite{simpson2}*{II.6.5}.

\smallskip

Secondly, we recall that $\NIN_{[0,1]}$ is not provable in $\Z_{2}^{\omega}+\IND_{0}$ by \cite{dagsamXI}*{Theorem~2.16}. 
To establish the theorem, we show that all items imply $\NIN_{[0,1]}$ over $\ACAo+\IND_{0}$.  To this end, let $Y:[0, 1]\di \N$ be an injection, fix continuous $f:[0,1]\di \R$,
and define $C_{n}:=\{ x\in [0,1]: Y(x)\leq n \}$.  Using $\IND_{0}$, one proves that each $C_{n}$ is closed (because finite), and $[0,1]=\cup_{n\in \N}C_{n}$ by assumption.
By definition, $f_{\upharpoonright C_{n}}$ is in $\AC(C_{n})$ and $\AC^{*}(C_{n})$, again using $\IND_{0}$, for $n\in \N$.  
Hence, we have shown that $C([0,1])$ equals $\ACG$ and $\ACG^{*}$, which yields a contradiction with the first two items, as Weierstrass' function is nowhere differentiable. 
The last two items similarly yield a contradiction, and we are done. 
\end{proof}
We note that the notion of \emph{generalised bounded variation} (\cite{gordonkordon}) can be treated in the same way as in the previous theorem.
Similarly, one can treat Gordon's classes \textsf{ACGs} and $\ACG_{\delta}$ from \cites{gordon1, gordonkordon}. 

\begin{ack}\rm 
We thank Anil Nerode for his valuable advice, especially the suggestion of studying nsc for the Riemann integral, and discussions related to Baire classes.
We thank Dave R.\ Renfro for his efforts in providing a most encyclopedic summary of analysis, to be found online.  
Our research was supported by the \emph{Deutsche Forschungsgemeinschaft} via the DFG grant SA3418/1-1 and the \emph{Klaus Tschira Boost Fund} via the grant Projekt KT43.
We express our gratitude towards the aforementioned institutions.    
\end{ack}
%


\appendix
\section{Technical Appendix: introducing Reverse Mathematics}\label{RMA}
We discuss Reverse Mathematics (Section \ref{introrm}) and introduce -in full detail- Kohlenbach's base theory of \emph{higher-order} Reverse Mathematics (Section \ref{rmbt}).
Some common notations may be found in Section \ref{kkk}.
\subsection{Introduction}\label{introrm}
Reverse Mathematics (RM for short) is a program in the foundations of mathematics initiated around 1975 by Friedman (\cites{fried,fried2}) and developed extensively by Simpson (\cite{simpson2}).  
The aim of RM is to identify the minimal axioms needed to prove theorems of ordinary, i.e.\ non-set theoretical, mathematics. 

\smallskip

We refer to \cite{stillebron} for a basic introduction to RM and to \cite{simpson2, simpson1,damurm} for an overview of RM.  We expect basic familiarity with RM, but do sketch some aspects of Kohlenbach's \emph{higher-order} RM (\cite{kohlenbach2}) essential to this paper, including the base theory $\RCAo$ (Definition \ref{kase}).  

\smallskip

First of all, in contrast to `classical' RM based on \emph{second-order arithmetic} $\Z_{2}$, higher-order RM uses a subset of $\L_{\omega}$, the richer language of \emph{higher-order arithmetic}.  
Indeed, while the former is restricted to natural numbers and sets of natural numbers, higher-order arithmetic can accommodate sets of sets of natural numbers, sets of sets of sets of natural numbers, et cetera.  
To formalise this idea, we introduce the collection of \emph{all finite types} $\mathbf{T}$, defined by the two clauses:
\begin{center}
(i) $0\in \mathbf{T}$   and   (ii)  If $\sigma, \tau\in \mathbf{T}$ then $( \sigma \di \tau) \in \mathbf{T}$,
\end{center}
where $0$ is the type of natural numbers, and $\sigma\di \tau$ is the type of mappings from objects of type $\sigma$ to objects of type $\tau$.
In this way, $1\equiv 0\di 0$ is the type of functions from numbers to numbers, and  $n+1\equiv n\di 0$.  Viewing sets as given by characteristic functions, we note that $\Z_{2}$ only includes objects of type $0$ and $1$.    

\smallskip

Secondly, the language $\L_{\omega}$ includes variables $x^{\rho}, y^{\rho}, z^{\rho},\dots$ of any finite type $\rho\in \mathbf{T}$.  Types may be omitted when they can be inferred from context.  
The constants of $\L_{\omega}$ include the type $0$ objects $0, 1$ and $ <_{0}, +_{0}, \times_{0},=_{0}$  which are intended to have their usual meaning as operations on $\N$.
Equality at higher types is defined in terms of `$=_{0}$' as follows: for any objects $x^{\tau}, y^{\tau}$, we have
\be\label{aparth}
[x=_{\tau}y] \equiv (\forall z_{1}^{\tau_{1}}\dots z_{k}^{\tau_{k}})[xz_{1}\dots z_{k}=_{0}yz_{1}\dots z_{k}],
\ee
if the type $\tau$ is composed as $\tau\equiv(\tau_{1}\di \dots\di \tau_{k}\di 0)$.  
Furthermore, $\L_{\omega}$ also includes the \emph{recursor constant} $\mathbf{R}_{\sigma}$ for any $\sigma\in \mathbf{T}$, which allows for iteration on type $\sigma$-objects via generalisations of the defining axiom in item \eqref{poliop} of Definition \ref{kase}.  Formulas and terms are defined as usual.  
One obtains the sub-language $\L_{n+2}$ by restricting the above type formation rule to produce only type $n+1$ objects (and related types of similar complexity).        

\subsection{The base theory of higher-order Reverse Mathematics}\label{rmbt}
We introduce Kohlenbach's base theory $\RCAo$, first introduced in \cite{kohlenbach2}*{\S2}.
\bdefi\label{kase} 
The base theory $\RCAo$ consists of the following axioms.
\begin{enumerate}
 \renewcommand{\theenumi}{\alph{enumi}}
\item  Basic axioms expressing that $0, 1, <_{0}, +_{0}, \times_{0}$ form an ordered semi-ring with equality $=_{0}$.
\item Basic axioms defining the well-known $\Pi$ and $\Sigma$ combinators \(aka $K$ and $S$ in \cite{avi2}\), which allow for the definition of \emph{$\lambda$-abstraction}. 
\item \label{poliop} The defining axiom of the recursor constant $\mathbf{R}_{0}$: for $m^{0}$ and $f^{1}$: 
\be\label{special}
\mathbf{R}_{0}(f, m, 0):= m \textup{ and } \mathbf{R}_{0}(f, m, n+1):= f(n, \mathbf{R}_{0}(f, m, n)).
\ee
\item The \emph{axiom of extensionality}: for all $\rho, \tau\in \mathbf{T}$, we have:
\be\label{EXT}\tag{$\textsf{\textup{E}}_{\rho, \tau}$}  
(\forall  x^{\rho},y^{\rho}, \varphi^{\rho\di \tau}) \big[x=_{\rho} y \di \varphi(x)=_{\tau}\varphi(y)   \big].
\ee 
\item The induction axiom for quantifier-free formulas of $\L_{\omega}$.
\item $\QFAC^{1,0}$: the quantifier-free Axiom of Choice as in Definition \ref{QFAC}.
\end{enumerate}
\edefi
\noindent
Note that variables (of any finite type) are allowed in quantifier-free formulas of the language $\L_{\omega}$: only quantifiers are banned.
G\"odel called the functionals obtained from $\mathbf{R}_\rho$ for all $\rho\in\mathbf{T}$ the \emph{primitive recursive functionals of finite type}.  Their theory
is axiomatised by \emph{G\"odel's system $T$}, studied in detail in \cite{avi2}*{Section 2.2}. 
\bdefi\label{QFAC} The axiom $\QFAC$ consists of the following for all $\sigma, \tau \in \textbf{T}$:
\be\tag{$\QFAC^{\sigma,\tau}$}
(\forall x^{\sigma})(\exists y^{\tau})A(x, y)\di (\exists Y^{\sigma\di \tau})(\forall x^{\sigma})A(x, Y(x)),
\ee
for any quantifier-free formula $A$ in the language of $\L_{\omega}$.
\edefi
As discussed in \cite{kohlenbach2}*{\S2}, $\RCAo$ and $\RCA_{0}$ prove the same sentences `up to language' as the latter is set-based and the former function-based.   
This conservation results is obtained via the so-called $\ECF$-interpretation, which we now discuss. 
\begin{rem}[The $\ECF$-interpretation]\rm
The (rather) technical definition of $\ECF$ may be found in \cite{troelstra1}*{p.\ 138, \S2.6}.
Intuitively, the $\ECF$-interpretation $[A]_{\ECF}$ of a formula $A\in \L_{\omega}$ is just $A$ with all variables 
of type two and higher replaced by type one variables ranging over so-called `associates' or `RM-codes' (see \cite{kohlenbach4}*{\S4}); the latter are (countable) representations of continuous functionals.  
The $\ECF$-interpretation connects $\RCAo$ and $\RCA_{0}$ (see \cite{kohlenbach2}*{Prop.\ 3.1}) in that if $\RCAo$ proves $A$, then $\RCA_{0}$ proves $[A]_{\ECF}$, again `up to language', as $\RCA_{0}$ is 
formulated using sets, and $[A]_{\ECF}$ is formulated using types, i.e.\ using type zero and one objects.  
\end{rem}
In light of the widespread use of codes in RM and the common practise of identifying codes with the objects being coded, it is no exaggeration to refer to $\ECF$ as the \emph{canonical} embedding of higher-order into second-order arithmetic.

\subsection{Notations and the like}\label{kkk}
We introduce the usual notations for common mathematical notions, like real numbers, as also introduced in \cite{kohlenbach2}.  
\begin{defi}[Real numbers and related notions in $\RCAo$]\label{keepintireal}\rm~
\begin{enumerate}
 \renewcommand{\theenumi}{\alph{enumi}}
\item Natural numbers correspond to type zero objects, and we use `$n^{0}$' and `$n\in \N$' interchangeably.  Rational numbers are defined as signed quotients of natural numbers, and `$q\in \Q$' and `$<_{\Q}$' have their usual meaning.    
\item Real numbers are coded by fast-converging Cauchy sequences $q_{(\cdot)}:\N\di \Q$, i.e.\  such that $(\forall n^{0}, i^{0})(|q_{n}-q_{n+i}|<_{\Q} \frac{1}{2^{n}})$.  
We use Kohlenbach's `hat function' from \cite{kohlenbach2}*{p.\ 289} to guarantee that every $q^{1}$ defines a real number.  
\item We write `$x\in \R$' to express that $x^{1}:=(q^{1}_{(\cdot)})$ represents a real as in the previous item and write $[x](k):=q_{k}$ for the $k$-th approximation of $x$.    
\item Two reals $x, y$ represented by $q_{(\cdot)}$ and $r_{(\cdot)}$ are \emph{equal}, denoted $x=_{\R}y$, if $(\forall n^{0})(|q_{n}-r_{n}|\leq {2^{-n+1}})$. Inequality `$<_{\R}$' is defined similarly.  
We sometimes omit the subscript `$\R$' if it is clear from context.           
\item Functions $F:\R\di \R$ are represented by $\Phi^{1\di 1}$ mapping equal reals to equal reals, i.e.\ extensionality as in $(\forall x , y\in \R)(x=_{\R}y\di \Phi(x)=_{\R}\Phi(y))$.\label{EXTEN}
\item The relation `$x\leq_{\tau}y$' is defined as in \eqref{aparth} but with `$\leq_{0}$' instead of `$=_{0}$'.  Binary sequences are denoted `$f^{1}, g^{1}\leq_{1}1$', but also `$f,g\in C$' or `$f, g\in 2^{\N}$'.  Elements of Baire space are given by $f^{1}, g^{1}$, but also denoted `$f, g\in \N^{\N}$'.
\item For a binary sequence $f^{1}$, the associated real in $[0,1]$ is $\r(f):=\sum_{n=0}^{\infty}\frac{f(n)}{2^{n+1}}$.\label{detrippe}
\item Sets of type $\rho$ objects $X^{\rho\di 0}, Y^{\rho\di 0}, \dots$ are given by their characteristic functions $F^{\rho\di 0}_{X}\leq_{\rho\di 0}1$, i.e.\ we write `$x\in X$' for $ F_{X}(x)=_{0}1$. \label{koer} 
\end{enumerate}
\end{defi}
For completeness, we list the following notational convention for finite sequences.  
\begin{nota}[Finite sequences]\label{skim}\rm
The type for `finite sequences of objects of type $\rho$' is denoted $\rho^{*}$, which we shall only use for $\rho=0,1$.  
Since the usual coding of pairs of numbers goes through in $\RCAo$, we shall not always distinguish between $0$ and $0^{*}$. 
Similarly, we assume a fixed coding for finite sequences of type $1$ and shall make use of the type `$1^{*}$'.  
In general, we do not always distinguish between `$s^{\rho}$' and `$\langle s^{\rho}\rangle$', where the former is `the object $s$ of type $\rho$', and the latter is `the sequence of type $\rho^{*}$ with only element $s^{\rho}$'.  The empty sequence for the type $\rho^{*}$ is denoted by `$\langle \rangle_{\rho}$', usually with the typing omitted.  

\smallskip

Furthermore, we denote by `$|s|=n$' the length of the finite sequence $s^{\rho^{*}}=\langle s_{0}^{\rho},s_{1}^{\rho},\dots,s_{n-1}^{\rho}\rangle$, where $|\langle\rangle|=0$, i.e.\ the empty sequence has length zero.
For sequences $s^{\rho^{*}}, t^{\rho^{*}}$, we denote by `$s*t$' the concatenation of $s$ and $t$, i.e.\ $(s*t)(i)=s(i)$ for $i<|s|$ and $(s*t)(j)=t(|s|-j)$ for $|s|\leq j< |s|+|t|$. For a sequence $s^{\rho^{*}}$, we define $\overline{s}N:=\langle s(0), s(1), \dots,  s(N-1)\rangle $ for $N^{0}<|s|$.  
For a sequence $\alpha^{0\di \rho}$, we also write $\overline{\alpha}N=\langle \alpha(0), \alpha(1),\dots, \alpha(N-1)\rangle$ for \emph{any} $N^{0}$.  By way of shorthand, 
$(\forall q^{\rho}\in Q^{\rho^{*}})A(q)$ abbreviates $(\forall i^{0}<|Q|)A(Q(i))$, which is (equivalent to) quantifier-free if $A$ is.   
\end{nota}

\bye